\documentclass{article}
\usepackage{amsmath}
\usepackage{amsfonts}
\usepackage{amssymb}
\usepackage{amsthm}
\numberwithin{equation}{section}
\usepackage{float}
\usepackage[utf8]{inputenc}
\usepackage[english]{babel} 
\usepackage[T1]{fontenc}
\usepackage{empheq}
\usepackage{ifthen}
\usepackage{intcalc}
\usepackage{hyperref}
\usepackage{subcaption}     
\usepackage{multirow}
\usepackage[table]{xcolor}
\usepackage{booktabs}
\usepackage{authblk}

\usepackage{tikz}
\usepackage{smartdiagram}
\usetikzlibrary{arrows,shapes,positioning,shadows,trees,chains}

\usepackage{algpseudocode}
\usepackage[linesnumbered,boxed, ruled]{algorithm2e}
\SetKwRepeat{Do}{do}{while}%
\newtheorem{theorem}{Theorem}
\newtheorem{lemma}{Lemma}
\newtheorem{definition}{Definition}
\newtheorem{proposition}{Proposition}
\newtheorem{corollary}{Corollary}
\newtheorem{remark}{Remark}
\newtheorem{assumption}{Assumption}

\numberwithin{theorem}{section}
\numberwithin{lemma}{section}
\numberwithin{remark}{section}
\numberwithin{definition}{section}
\numberwithin{proposition}{section}
\numberwithin{corollary}{section}
\numberwithin{assumption}{section}

\newcommand{\ds}{\displaystyle}
\newcommand{\pa}{\partial}
\def\eps{{\varepsilon}}
\newcommand{\R}{\mathbb{R}}
\newcommand{\dive}{{\rm div}\,} 
 
\newcommand{\grad}{{\bf grad}\,}

\def\avec{{\bf a}}

\def\nvec{{\bf n}}
\def\pvec{{\bf p}}
\def\qvec{{\bf q}}

\def\xvec{{\bf x}}
\def\yvec{{\bf y}}
\def\zvec{{\bf z}}
\def\Hvec{{\bf H}}

\def\PQvec{{\bf \tilde{P}Q}}
\def\PQvecmg{{\underline {\underline {\bf \tilde{P}Q}}}}
\def\Lvec{{\bf L}}
\def\Qvec{{\bf Q}}

\def\Xcal{{\cal X}}

\def\Ical{{\bf \mathcal I}}

\def\Pcal{{\bf \mathcal P}}

\def\Tcal{{\bf \mathcal T}}

\def\uelt{\mathtt{u}}
\def\welt{\mathtt{w}}
\def\Wens{\mathtt{W}}

\def\Nhat{{\widehat N}}
\def\Hhat{{\widehat \H}}

\def\H{{\mathbb H}}
\def\T{{\mathbb T}}
\def\D{{\mathbb D}}
\def\Sd{{\mathbb S}}
\newcommand{\N}{\mathbb{N}}
\def\Vudud{{\underline{\underline V}}}
\def\Wudud{{\underline{\underline W}}}
\def\Hudud{{\underline {\underline H}}}
\def\Ludud{{\underline {\underline L}}}
\def\Mudud{{\underline {\underline M}}}
\def\Lududvec{{\bf {\underline {\underline L}}}}

\def\Qududvec{{\bf {\underline  {\underline Q} } }}
\def\Vudud{{\underline {\underline V}}}

\def\Dfrak{{\mathfrak D}}
\def\Qududvec{{\bf {\underline  {\underline Q} } }}
\def\Xcalmg{{\underline{\underline{\cal X}}}}
\def\Pudud{\underline{\underline{{\mathbb P}}}}
\def\Lambdaudud{{\underline {\underline \Lambda}}}

\def\Tcoer{\mathtt{T}}

\tikzset{
  basic/.style  = {draw, text width=2cm, drop shadow, font=\sffamily, rectangle},
  root/.style   = {basic, rounded corners=2pt, thin, align=center,
                   fill=orange!50!red},
  level 2/.style = {basic, rounded corners=6pt, thin,align=center, fill=red!50,
                   text width=8em},
  level 3/.style = {basic, thin, align=left, fill=yellow!100, text width=6.5em}
}
 \tikzstyle{decision} = [diamond, draw, fill=yellow!50!orange, 
     text width=5em, text badly centered, node distance=2cm, inner sep=0pt,minimum height=1em]
   \tikzstyle{decision2} = [diamond, aspect=2, draw, fill=yellow!50!orange, 
     text width=8em, text badly centered, node distance=2cm, inner sep=0pt,minimum height=1em]  
 \tikzstyle{block} = [rectangle, draw, fill=blue!50, 
     text width=7em, text centered, rounded corners, minimum height=2em]
 \tikzstyle{line} = [draw, -latex']
 \tikzstyle{cloud} = [draw, ellipse,fill=pink!50,text width=5em,text badly centered, node distance=3cm,
     minimum height=2em]
 \tikzstyle{cloud2} = [draw, ellipse,fill=white, node distance=3cm,
 minimum height=2em]    

 \tikzstyle{cloud3} = [draw, circle, node distance=0.1cm,
 minimum height=0.001em]   
  \tikzstyle{block2} = [rectangle, draw, fill=green!50, 
  text width=12em, text centered, rounded corners, minimum height=2em]
   \tikzstyle{block3} = [rectangle, draw, fill=red!50, 
   text width=2.5em, text centered, rounded corners, minimum height=2em]
\tikzstyle{block4} = [rectangle, draw, fill=blue!50, 
     text width=13em, text badly centered, rounded corners, minimum height=2em]
\tikzstyle{block5} = [rectangle, draw, fill=blue!50, 
     text width=7em, text badly centered, rounded corners, minimum height=2em]
\tikzstyle{block6} = [rectangle, draw, fill=cyan!50, 
text width=7em, text badly centered, rounded corners, minimum height=2em]

\begin{document}
\title{A posteriori error estimates for mixed finite element discretization of the multigroup Neutron Simplified Transport equations with Robin boundary condition}
\author[1]{Patrick Ciarlet}
\author[2]{Minh Hieu Do}
\author[2]{Mario Gervais}
\author[2]{Fran\c{c}ois Madiot}
\affil[1]{POEMS, CNRS, INRIA, ENSTA, Institut Polytechnique de Paris, 91120 Palaiseau, France.}
\affil[2]{Universit\'e Paris-Saclay, CEA, Service d'\'Etudes des R\'eacteurs et de Math\'ematiques Appliqu\'ees, 91191, Gif-sur-Yvette, France.}

\maketitle

\begin{abstract}
 We analyse \emph{a posteriori} error estimates for the discretization with mixed finite elements on simplicial or Cartesian meshes of the multigroup neutron simplified transport (SP$_N$) equations, in the case where a Robin (or Fourier type) boundary condition is imposed on the boundary. This boundary condition is of particular importance in neutronics, since it corresponds to the well-known vacuum boundary condition. We provide guaranteed and locally efficient estimators. 
In particular, a specific estimator is designed to handle the Robin boundary condition. We also develop the theory in the case of mixed imposed boundary conditions, of Dirichlet, Neumann or Fourier type. 
The approach is further extended to a Domain Decomposition Method, the so-called DD+$L^2$ jumps method. In this framework, the adaptive mesh refinement strategy is implemented for a discretization using Cartesian meshes on each subdomain. Numerical experiments illustrate the theory. 
\end{abstract}
\section*{Introduction}
In neutronics, one is interested in modelling the neutron density inside a reactor core. 
The neutron flux density in the reactor core is determined by solving the transport equation which depends on seven variables: space (3), direction (2), energy or modulus of the velocity (1), and  time (1). 
Due to the high dimensionality of the problem, the numerical resolution of this equation faces some challenges in terms of computational cost. In practice, the neutron flux density can be modeled by the simplified transport equations~\cite{Gelb60} at the reactor core scale. 
The $SP_N$ equations stems from a modelling approximation of the transport equation.
Consequently, the $SP_N$ equations do not converge to transport equation. Nevertheless, they are commonly used by physicists since their resolution is cheap in terms of computational cost. The order $N$ is odd, and the number of $SP_N$ odd (resp. even) moments is $\Nhat:=\frac{N+1}{2}$. 
The energy variable is commonly discretized using the multigroup theory {\cite{DuHa76,marchuk1986numerical}}. In this method, the entire range of neutron energies is divided into $G$ intervals, called energy groups. In each energy group, the neutron flux density is lumped and all parameters are averaged. 

This model has the same structure as a multigroup neutron diffusion equations~\cite{JM21}. 
The numerical analysis of the multigroup SP$_N$ equations with a source term, discretized with mixed finite elements, may be found in~\cite{Giret2018,Gerv26}.
The analysis included in particular the case of low-regularity solutions. {\em A priori estimates} were derived in the process. A natural question is then the {\em a posteriori} analysis of the method, to further optimize the cost of the numerical method. {This question has been addressed for the neutron diffusion equations, with vanishing Dirichlet boundary condition~\cite{CDM23,CDM25}.
The main topic of this paper is to extend this approach to another model, the multigroup SP$_N$ equations, supplemented by another boundary condition, namely the Robin boundary condition.}

\par
{\em A posteriori} analysis for mixed finite elements has been extensively studied,  see~\cite{carstensen1997posteriori,larson2008posteriori,lovadina2006energy,Vohralik2010} and references therein for the Poisson equation,~{\cite{wohlmuth1999comparison,wheeler2005posteriori} for the diffusion-reaction equation} (one-group neutron diffusion equation), and ~\cite{Vohralik2007} for the convection-diffusion-reaction equation. {In~\cite{CDM23}, the first two authors and the last author proposed {\em a posteriori} estimators for the one-group neutron diffusion equation that are both reliable and locally efficient with respect to two norms to measure the errors. This study was performed in the setting of a Dirichlet boundary condition imposed on the boundary. The approach is generalized to the multigroup neutron diffusion equations in~\cite{CDM25}, see references therein. Extending this approach to the case of the Robin boundary condition is of particular importance for the neutron simplified transport model since it corresponds to the well-known vacuum boundary condition. A posteriori analysis for mixed finite elements for the Poisson problem with Robin boundary condition has been studied in~\cite{KSS11,L24}. In~\cite{L24}, the author provides estimators that are locally efficient. These estimators are however not reliable, in the sense that the upper bound on the error between the exact and approximate solutions is not fully computable: it depends on a generic constant which is independent of the mesh size.} Let us also mention some works on a posteriori estimates for the mortar mixed finite element method~\cite{wohlmuth1999comparison,wheeler2005posteriori,PVW2013}.

\par
Nuclear reactor cores often have a Cartesian geometry. Indeed, in the models, the base brick, which is called a cell, is a rectangular cuboid of $\R^3$. The global layout is a set of cells that are distributed on a 3D grid, so that the global domain of the reactor core can be represented by a rectangular cuboid of $\R^3$.  Each cell is made of fuel, absorbing or reflector material. To account for the different materials, the coefficients in the models are  {\em piecewise polynomials} (possibly piecewise constant) with respect to the position, ie. their restriction to each cell is a polynomial \cite{DuHa76,JaBL12,JaCi13}. In practice the coefficients characterizing the materials may differ from one cell to another by a factor of order $10$ or more.

\par
The outline of the manuscript is as follows. \\
In Sections~\ref{sec-notations} and~\ref{sec-model}, we introduce some notations and our model problem. 
Then in Section~\ref{sec-plain}, we recall how it can be solved in a mixed setting. To that aim we build the  standard equivalent variational formulation, and provide the existing {\em a priori} numerical analysis results that allow one to compare the discrete solution to the exact one. For the discretization, we choose the well-known Raviart-Thomas-N\'ed\'elec finite element RTN${}_k$, where $k\ge0$ denotes the order. In Section~\ref{section:error-estimator}, we propose the {\em a posteriori} analysis of the model. We build a reconstruction of the solution via an averaging, or a post-processing technique. We also investigate how the {\em a posteriori analysis} can be extended to a multi-domain reformulation (the so called DD+$L^2$-jumps method~\cite{CiJK17}) of our model problem. In Section~\ref{sec:numerical_results}, we illustrate numerically the theoretical results. 

\section{Notations}\label{sec-notations}
We choose the same notations as in \cite{CDM23}.
Throughout the paper, $C$ is used to denote a generic positive constant which is independent of the mesh size, the mesh and the quantities/fields of interest. We also use the shorthand notation $A\lesssim B$ for the inequality $A\leq C B$, where $A$ and $B$ are two scalar quantities, and $C$ is a generic constant. \\ 
Vector-valued (resp. tensor-valued) function spaces are written in boldface character (resp. blackboard characters)\,; for the latter, the index \textit{sym} indicates symmetric fields. Given an open set ${\mathcal O}\subset\R^{{d}}$, ${d}=1,2,3$, we use the notation $(\cdot,\cdot)_{0,{\mathcal O}}$ (respectively $\|\cdot\|_{0,{\mathcal O}}$) for the $L^2({\mathcal O})$ and $\Lvec^2({\mathcal O})=(L^2({\mathcal O}))^{{d}}$ scalar products (resp. norms). More generally, $(\cdot,\cdot)_{s,{\mathcal O}}$ and $\|\cdot\|_{s,{\mathcal O}}$ (respectively $|\cdot|_{s,{\mathcal O}}$) denote the scalar product and norm (resp. semi-norm) of the Sobolev spaces $H^s({\mathcal O})$ and $\Hvec^s({\mathcal O})=(H^s({\mathcal O}))^{{d}}$ for $s\in\R$ (resp. for $s>0$). \\ 
If moreover the boundary $\pa {\mathcal O}$ is Lipschitz, $\nvec$ denotes the unit outward normal vector field to $\pa {\mathcal O}$. Finally, it is assumed that the reader is familiar with vector-valued function spaces related to the diffusion equation, such as $\Hvec(\dive;{\mathcal O})$, $\Hvec_0(\dive;{\mathcal O})$ etc. \\
Let $G\in\N\setminus\{0,1\}$
and $\Nhat\in\N\setminus\{0,1\}$. Given a function space $W$, we denote by $\Wudud$ the product space $W^{\Nhat\times G}$.
We extend the notation $(\cdot,\cdot)_{0,{\mathcal O}}$ (respectively $\|\cdot\|_{0,{\mathcal O}}$) to the $\underline{\underline{\Lvec^2}}({\mathcal O})$ and $\underline{\underline{\Lvec^2}}({\mathcal O})$ inner products (resp. norms).\\
Specifically, we let $\Omega$ be a bounded, connected and open subset of $\R^d$ {for $d=2,3$}, having a Lipschitz boundary which is piecewise smooth. We split $\Omega$ into $N$ open, connected, disjoint parts $\{\Omega_i\}_{1\le i\le N}$ with Lipschitz, piecewise smooth boundaries: $\overline{\Omega}=\cup_{1\le i\le N}\overline{\Omega_i}$ and the set $\{\Omega_i\}_{1\le i\le N}$ is called a partition of $\Omega$. For a field $v$ defined over $\Omega$, we shall use the notations $v_i=v_{|\Omega_i}$, for $1\le i\le N$.\\
Given a partition $\{\Omega_i\}_{1\le i\le N}$ of $\Omega$, we introduce a function space with piecewise regular elements:
\[
\begin{array}{rcl}
\Pcal W^{1,\infty}(\Omega)&=&\left\{D\in L^{\infty}(\Omega)\,|\,D_i\in W^{1,\infty}(\Omega_i),\,1\le i\le N\right\}.
\end{array}
\]
To measure $\psi\in\Pcal W^{1,\infty}(\Omega)$, we use the natural norm $$\|\psi\|_{\Pcal W^{1,\infty}(\Omega)} =\max_{i=1,N}\|\psi_i\|_{W^{1,\infty}(\Omega_i)}.$$

\section{The model}\label{sec-model}
We introduce the model such as defined in~\cite{BaLa11}. Let us set $\Ical_G:=\{1,\cdots,G\}$, the set of energy group indices. We denote by $\Ical_e$ (resp. $\Ical_o$) the subset of even (resp. odd) integers of the integer set $\{0,\cdots, N \}$.
Given a source term $S_f\in \Ludud^2(\Omega)$, we consider the following neutron SP$_N$ equations, with vacuum boundary condition. In its primal form, it is written:
\begin{equation}\label{eq:diff_primal}
\left\{\begin{array}{l}
\mbox{Find $\phi\in \Hudud^1(\Omega)$ such that}\cr
-\dive (\Dfrak\,\grad\phi )+ \T_e\,\phi= S_f\mbox{ in }\Omega,\cr
(\Dfrak\,\grad\phi )\cdot \nvec + \Gamma_e \phi = 0 \mbox{ in }\partial\Omega.
\end{array}\right.
\end{equation}
where $\phi$ and $S_f$ denote respectively the neutron flux and the fission source, and $\Dfrak=\H^T\,\T_o^{-1}\,\H$.
Let $\delta_{\cdot,\cdot}$ be the Kronecker symbol. The matrices $\H$, $\T_e$, $\T_o, \\Gamma_e\in\left(\R^{\Nhat\times\Nhat}\right)^{G\times G}$
 are such that $\forall (g,g')\in\Ical_G\times\Ical_G$ :
\begin{itemize}
\item $(\H)_{g,g'}=\delta_{g,g'}\Hhat\in\R^{\Nhat\times\Nhat}$, with $\forall(i,j)\in\{1,\cdots,\Nhat\}^2$, $\Hhat_{i,j}=\delta_{i,j}+\delta_{i,j-1}$.
\item $(\T_e)_{g,g}:=\T_e^g\in\R^{\Nhat\times\Nhat}$ denotes the even removal matrix, such that:
\[
\T_e^g=diag\,
\left(t_{m}(\Sigma_t^g-\Sigma_{s,m}^{g\to g})\right)_{m\in\Ical_e},
\]
\item[]$(\T_o)_{g,g}:=\T_o^g\in\R^{\Nhat\times\Nhat}$ denotes the odd removal matrix, such that: 
\[
\T_o^g=diag\,
\left(t_{m}(\Sigma_t^g-\Sigma_{s,m}^{g\to g})\right)_{m\in\Ical_o},
\] 
\item[]
where $\forall m\in\Ical_{e,o}$, 
 $\forall m\geq 0$, $t_m=\frac{(\alpha_m)^2}{2m+1}>0$ with $\alpha_0=1$ and $\alpha_{m+1}=\frac{4(m+1)^2-1}{(m+1)\alpha_m}$.
\item[]
The coefficient $\Sigma_t^g$ is the macroscopic total cross section of energy group $g$, and the coefficient
 $\Sigma_{s,m}^{g\to g}$ denotes the Legendre moment of order $m$ of the macroscopic self scattering cross sections, from energy group $g$ to itself.
\item For $g'\neq g$:
\item[]$(\T_e)_{g,g'}:=-\Sd_e^{g'\rightarrow g}\in\R^{\Nhat\times\Nhat}$ denotes the even scattering matrix, such that: 
\[
\Sd_e^{g'\rightarrow g}=diag\,
\left(t_{m}\Sigma_{s,m}^{g'\rightarrow g}\right)_{m\in\Ical_e},
\]
\item[]$(\T_o)_{g,g'}:=-\Sd_o^{g'\rightarrow g}\in\R^{\Nhat\times\Nhat}$ denotes the odd scattering matrix, such that: 
\[
\Sd_o^{g'\rightarrow g}=diag\,
\left(t_{m}\Sigma_{s,m}^{g'\rightarrow g}\right)_{m\in\Ical_o},
\]
\item[]
where $\Sigma_{s,m}^{g'\rightarrow g}$ is the Legendre moment of order $m$ the macroscopic scattering cross sections from energy group $g'$ to energy group $g$.
\item $(\Gamma_e)_{g,g'}=\delta_{g,g'}\hat{\Gamma}_e\in\R^{\Nhat\times\Nhat}$, is a symmetric positive definite matrix where $\hat{\Gamma}_e$ is defined by $\forall(i,j)\in\{1,\cdots,\Nhat\}^2$, $$(\hat{\Gamma}_e)_{i,j}=\alpha_{2(i-1)}\alpha_{2(j-1)}(xP_{2(i-1)}(x),P_{2(j-1)}(x))_{0,(0,1)},$$ with $P_m$ the m\textsuperscript{th} Legendre polynomial.
\end{itemize}
The coefficients of the matrices $\T_{e,o}$ are supposed to be such that:
\begin{equation}\label{Pos-SPN}
\left\{
\begin{array}{ll}
(0)&
\forall\,g,\,g'\in\Ical_G,\forall\,m\in\Ical_{e,o}:\\
&\quad (\Sigma_{r,m}^g,\Sigma_{s,m}^{g'\to g})\in\Pcal W^{1,\infty}(\Omega)\times L^\infty(\Omega).\\
(i)&
\exists\,(\Sigma_{r,(e,o)})_*,\,(\Sigma_{r,(e,o)})^*>0\,|\,\forall\,g\in\Ical_G,\,\forall\,m\in\Ical_{e,o}:\\
&(\Sigma_{r,(e,o)})_*\le t_m\Sigma_{r,m}^g\le (\Sigma_{r,(e,o)})^*\mbox{ a.e. in }\Omega.\\
(ii)&
\exists\,0<\eps<\ds\frac{1}{G-1} \,|\,\forall\,m\in\Ical_{e,o},\,\forall\,g,g'\in\Ical_G, g'\neq g, 
\\&
\quad |\Sigma_{s,m}^{g\rightarrow g'}|\leq\eps \Sigma_{r,m}^g\mbox{ a.e. in }\Omega,\\
 (iii) & 
 \exists  (\T_o^{-1})^*,\,(\T_o^{-1})_*,\,(\T_e)^*,\,(\T_e)_*>0, \mbox{ such that }  \forall X \in \underline{\underline{\R}},\mbox{ a.e. in }\Omega:\\
 &\quad  \left\{\begin{array}{ll}
(\T_o^{-1})_*\|X\|^2\leq X^T\T_o^{-1}X, &\|\T_o^{-1}X\| \le (\T_o^{-1})^*\|X\|, \cr 
(\T_e)_*\|X\|^2 \le X^T\T_eX,                   & \|\T_eX\| \le (\T_e)^*\|X\|,
\end{array}\right.
\end{array}
\right.
\end{equation}
where  $\forall\,g \in\Ical_G,\forall\,m\in\Ical_{e,o}, \Sigma_{r,m}^g:=\Sigma_t^g-\Sigma_{s,m}^{g\to g}$.

We refer to~\cite[Section 1.5.3]{Giret2018} for the formulation of a set of necessary conditions under which \eqref{Pos-SPN}-(iii) holds true.
Hypothesis \ref{Pos-SPN}$-(ii)$ is valid while modelling the core of a pressurized water reactor:  the scattering cross-sections are weaker than the removal cross-sections of an order $0<\eps<<1$. 
Thus, the matrices $\T_{e,o}$ are strictly diagonally dominant matrices: in particular, they are  invertible, and so $\Dfrak$ is well-defined. \\
Starting from the assumption (\ref{Pos-SPN})(iii), one can prove easily that there exists $(\T_e^{-1})_*,\,(\T_e^{-1})^*,\,(\T_o)_*,\,(\T_o)^*>0$ such that for all $X \in \underline{\underline{\R}}$, almost everywhere in $\Omega$,
	\begin{equation}\label{eq:pos_SPN_mixed}
\left\{\begin{array}{ll}
	(\T_e^{-1})_*\|X\|^2 \leq X^T\T_e^{-1}X,  & \|\T_e^{-1}X\|\leq (\T_e^{-1})^*\|X\|,  \cr
	(\T_o)_*\|X\|^2 \leq  X^T\T_oX,                 & {\|\T_oX\|\leq (\T_o)^*\|X\|}.
\end{array}\right.
\end{equation}
Classically, Problem \eqref{eq:diff_primal} is equivalent to the following variational formulation:
\begin{equation}\label{eq:FV_diff_primal}
\left\{\begin{array}{l}
\mbox{Find $\phi\in \Hudud^1(\Omega)$ such that }\forall \psi\in \Hudud^1(\Omega),\cr
(\Dfrak\,\grad\phi,\grad\psi)_{0,\Omega} + (\T_e\phi,\psi)_{0,\Omega}{+ (\Gamma_e\phi,\psi)_{0,\pa\Omega}} = (S_f,\psi)_{0,\Omega}.
\end{array}\right.
\end{equation}
\noindent Under the assumptions \eqref{Pos-SPN} on the coefficients, the primal problem \eqref{eq:diff_primal} is well-posed, in the sense that for all $S_f\in \Ludud^2(\Omega)$, there exists one and only one solution $\phi\in {\Hudud^1(\Omega)}$ that solves \eqref{eq:diff_primal}, with the bound $\|\phi\|_{1,\Omega}\lesssim\,\|S_f\|_{0,\Omega}$. Provided that the coefficient $\Dfrak$ is piecewise smooth, the solution { has extra smoothness} (see eg. Proposition 1 in \cite{CiJK17}). 
Throughout the paper, we add remarks on the extension in the situation where $\T_e\geq 0$ may vanish. In particular, the {\em a posteriori} analysis we propose covers both the pure diffusion case, and the diffusion-reaction case. \\
For simplicity, we prescribe a Robin boundary condition everywhere on $\partial\Omega$. However, instead of imposing only a Robin boundary condition, one can consider mixed boundary conditions on $\pa\Omega$, in which case analyses can also be carried out theoretically and numerically. Results are detailed in Appendix~\ref{sec:appendix_mixed}.

\section{ Variational formulation and discretization}\label{sec-plain}
Let us introduce the function spaces:
\[\begin{array}{rcl}
\Qvec(\Omega) &=& \left\{\, \qvec\in\Hvec(\dive,\Omega) \, | (\qvec\cdot\nvec)_{|_{\pa\Omega}}\in L^2(\pa\Omega) \right\},\cr\\
&  &\|\qvec\|_{\Qvec(\Omega)}=\left(\|\qvec\|_{\Hvec(\dive,\Omega)}^2\,+\,
\|\qvec\cdot\nvec\|_{0,\pa \Omega}^2\right)^{1/2};\cr
\Xcal &=& \left\{\,(\qvec,\psi)\in\Qvec(\Omega)\times L^2(\Omega)\right\}\,,\ \|(\qvec,\psi)\|_{\Xcal}=\left(\|\qvec\|_{\Qvec(\Omega)}^2\,+\,\|\psi\|_{0,\Omega}^2\right)^{1/2}\,.
\end{array}\]
We also use the notations: $\zeta=(\pvec,\phi)$ and $\xi=(\qvec,\psi)$.
\subsection{Mixed variational formulation}\label{ss-sec:VF-PC}
The solution $\phi$ to \eqref{eq:diff_primal} belongs to ${\Hudud^1(\Omega)}$, so if one lets ${\pvec=-\T_o^{-1}\H\,\grad\phi}\in\Lududvec^2(\Omega)$, the neutron multigroup SP$_N$ problem may {also} be written as:
\begin{equation}\label{eq:diff-mixed}
\left\{\begin{array}{l}
\mbox{Find $(\pvec,\phi)\in{\Qududvec(\Omega)}\times {\Hudud^1(\Omega)}$ such that}\cr
\T_o\,\pvec\,+\,\H\grad\phi=0\mbox{ in }\Omega,\cr
\H^T\dive\pvec\,+\,\T_e\phi=S_{f}\mbox{ in }\Omega,\cr
{-\H^T\pvec\cdot \nvec + \Gamma_e\phi=0\mbox{ on }\pa\Omega.}
\end{array}\right.
\end{equation}
Solving the mixed problem \eqref{eq:diff-mixed} is equivalent to solving \eqref{eq:diff_primal}.
\begin{proposition}\label{pro:thm0-PC}
	{Let $\T_{o}$ and $\T_{e}$  satisfy \eqref{Pos-SPN}.} The solution $(\pvec,\phi)\in{\Qududvec(\Omega)}\times {\Hudud^1(\Omega)}$ to \eqref{eq:diff-mixed} is such that $\phi$ is a solution to \eqref{eq:diff_primal} with the same data. Conversely, the solution $\phi\in \Hudud^1(\Omega)$ to \eqref{eq:diff_primal} is such that $(-{\T_o^{-1}\H\,\grad\phi},\phi)\in{\Qududvec(\Omega)}\times {\Hudud^1(\Omega)}$ is a solution to \eqref{eq:diff-mixed} with the same data.
\end{proposition}
To obtain the variational formulation for the mixed problem \eqref{eq:diff-mixed}, let $\qvec\in{\Qududvec(\Omega)}$ and $\psi\in \Ludud^2(\Omega)$, multiply the first equation of \eqref{eq:diff-mixed} by $-\qvec$, the second equation of \eqref{eq:diff-mixed} by $\psi$, and integrate over $\Omega$. Adding up the contributions, one finds that:
\begin{equation}\label{eq:VF-0-PC}
-(\T_o\,\pvec,\qvec)_{0,\Omega}
-(\H\grad\phi,\qvec)_{0,\Omega}
+(\H^T\dive\pvec,\psi)_{0,\Omega}
+(\T_e\phi,\psi)_{0,\Omega}
= (S_{f},\psi)_{0,\Omega}.
\end{equation}
One may integrate by parts the second term in the left-hand side, which yields: {$-(\H\grad\phi,\qvec)_{0,\Omega} = (\phi,\H^T\dive\qvec)_{0,\Omega} - (\phi,{\H^T(\qvec\cdot\nvec)})_{0,\pa\Omega}$}.
Hence, the solution to \eqref{eq:diff-mixed} also solves {a variational formulation set in $\Xcalmg = \Qududvec(\Omega)\times \Ludud^2(\Omega)$}:
\begin{equation}\label{eq:VF-1-PC}
\left\{\begin{array}{l}
\mbox{Find $(\pvec,\phi)\in\Xcalmg$ such that }\forall(\qvec,\psi)\in\Xcalmg,\cr
 -(\T_o\,\pvec,\qvec)_{0,\Omega}+(\phi,\H^T\dive\qvec)_{0,\Omega}+(\H^T\dive\pvec,\psi)_{0,\Omega}\cr 
\qquad +(\T_e\,\phi,\psi)_{0,\Omega} {- (\tilde{\Gamma}_e(\pvec\cdot\nvec), (\qvec\cdot\nvec))_{0,\pa\Omega}}= (S_f,\psi)_{0,\Omega},
\end{array}\right.
\end{equation}
where $\tilde{\Gamma}_e=\H\Gamma_e^{-1}\H^T$ is a symmetric positive definite matrix.
Clearly, the form
\begin{align}
c\ :\ ((\pvec,\phi),(\qvec,\psi)) \mapsto 
&-(\T_o\,\pvec,\qvec)_{0,\Omega}
+(\phi,\H^T\dive\qvec)_{0,\Omega}
+(\psi,\H^T\dive\pvec)_{0,\Omega}\nonumber\\
&\quad +(\T_e\,\phi,\psi)_{0,\Omega}
{- (\tilde{\Gamma}_e(\pvec\cdot\nvec), (\qvec\cdot\nvec))_{0,\pa\Omega}},\label{eq:bilin-form-c-PC}
\end{align}
is a  continuous bilinear form on $\Xcalmg$.

We may rewrite the variational formulation \eqref{eq:VF-1-PC} as:
\begin{equation}\label{eq:VF-3-PC}
\left\{\begin{array}{l}
\mbox{Find $(\pvec,\phi)\in\Xcalmg$ such that}\cr
\forall(\qvec,\psi)\in\Xcalmg,\quad c((\pvec,\phi),(\qvec,\psi))={(S_{f},\psi)_{0,\Omega}}.
\end{array}\right.
\end{equation}
{The proof of the next result is classical (and omitted here).}
\begin{proposition}\label{pro:VF-EQ-PC}
	The solution $\zeta=(\pvec,\phi)$ to \eqref{eq:VF-3-PC} satisfies \eqref{eq:diff-mixed}. Hence, problems \eqref{eq:VF-3-PC} and \eqref{eq:diff-mixed} are equivalent.
\end{proposition}
One may prove that the mixed formulation \eqref{eq:VF-3-PC} is well-posed
using $\Tcoer$-coercivity, cf. section 1.2.2 in \cite{Ciar25}, i.e. one has to prove that
\begin{align*}
&\exists\alpha>0,\ \exists \Tcoer\in{\cal L}(\Xcalmg) \text{ bijective, such that }\\
&\forall (\pvec,\phi)\in \Xcalmg, \quad c((\pvec,\phi),\Tcoer(\pvec,\phi)) \ge \alpha \|(\pvec,\phi)\|_{\Xcalmg}^2.
\end{align*}
\begin{theorem}\label{th:VF-1-PC}
Let $\T_{o}$ and $\T_e$ satisfy \eqref{Pos-SPN}. Then, the bilinear form $c$ is $\Tcoer$-coercive.\end{theorem}
\begin{proof}
{We choose the map $\Tcoer$ in the spirit of {\cite[Theorem 3.16]{Giret2018}}.
Given $(\pvec,\phi)\in\Xcalmg$, we let $\Tcoer((\pvec,\phi)) = (-\pvec,\frac{1}{2}(\phi + \T_e^{-T}\H^T\dive\pvec))\in\Xcalmg$.
Obviously, one has $\Tcoer\in{\cal L}(\Xcalmg)$.
In addition, $\Tcoer$ is bijective. Indeed, injectivity is obvious, while given $(\qvec,\psi)\in\Xcalmg$, one checks that choosing $(\pvec,\phi) = (-\qvec, 2\psi + \T_e^{-T}\H^T\dive\qvec)\in\Xcalmg$ yields $\Tcoer((\pvec,\phi)) =(\qvec,\psi)$, so $\Tcoer$ is surjective as well. \\
While, according to the definition of the bilinear form $c$, we have
\begin{align*}
c((\pvec,\phi),\Tcoer(\pvec,\phi))& = (\T_o\,\pvec,\pvec)_{0,\Omega}+\frac{1}{2}(\T_e^{-T}\H^T\dive\pvec,\H^T\dive\pvec)_{0,\Omega} \nonumber\\ &\quad {+ (\tilde{\Gamma}_e(\pvec\cdot\nvec), (\pvec\cdot\nvec))_{0,\pa\Omega}}
+\frac{1}{2}(\T_e\,\phi,\phi )_{0,\Omega}\nonumber\\
&\geq (\T_o)_*\|\pvec\|_{0,\Omega}^2 + \frac{1}{2}(\H^T)_*(\T_e^{-1})_*\|\dive\pvec\|_{0,\Omega}^2\\
&\quad + (\tilde{\Gamma}_e)_* \|(\pvec\cdot\nvec)\|_{0,\partial\Omega}^2+ \frac{1}{2}(\T_e)_*\|\phi\|_{0,\Omega}^2,\nonumber\\
&\geq \min\left\{(\T_o)_*,\frac{1}{2}(\H^T)_*(\T_e^{-1})_*, (\tilde{\Gamma}_e)_*, \frac{1}{2}(\T_e)_*\right\}\|\zeta\|_{\Xcalmg}^2,
\end{align*}
where 
\begin{align*}
&(\H^T)_* = \inf_{X\in\underline{\underline{\R}}\setminus\{0\}}\frac{\|\H^TX\|^2}{\|X\|^2}>0,\ 
(\H^T)^* = \sup_{X\in\underline{\underline{\R}}\setminus\{0\}}\frac{\|\H^TX\|}{\|X\|}>0,\ \\ 
&(\tilde{\Gamma}_e)_* = \inf_{X\in\underline{\underline{\R}}\setminus\{0\}}\frac{X^T\tilde{\Gamma}_e X}{\|X\|^2}>0.
\end{align*}
Hence, the form $c$ is $\Tcoer$-coercive.
}
\end{proof}
\subsection{Discretization and {\em a priori} error analysis}\label{ss-sec-FEM-PC}
We study conforming discretizations of \eqref{eq:VF-3-PC}. Let $(\Tcal_h)_h$ be a family of { meshes, made for instance of simplices, or of rectangles ($d=2$), resp. cuboids ($d=3$)}, indexed by a parameter $h$ equal to the largest diameter of elements of a given {mesh}. Let us introduce some further notations, given such a mesh $\Tcal_h$.
The set of facets of $\Tcal_h$ is denoted $\mathcal{F}_h$, and it is split as $\mathcal{F}_h = \mathcal{F}_h^i \cup \mathcal{F}_h^e$, with $\mathcal{F}_h^e$ (resp. $\mathcal{F}_h^i$) being the set of boundary facets (resp. interior facets). Given $K \in \Tcal_h$, for all faces $F\in \mathcal{F}_h\cap \partial K$, we denote 
$\nvec_F$ the unit outward normal to the face $F$. 
 We introduce discrete, finite-dimensional, spaces indexed by $h$ as follows:
\[ \Qvec_h\subset{\Qvec(\Omega)},\mbox{ and }L_h\subset L^2(\Omega). \]
The conforming discretization of the variational formulation \eqref{eq:VF-3-PC} is then:
\begin{equation}\label{eq:VF-1h-PC}
\left\{\begin{array}{l}
\mbox{Find $(\pvec_h,\phi_h)\in\Qududvec{}_h\times \Ludud{}_h$ such that}\cr
\forall(\qvec_h,\psi_h)\in\Qududvec{}_h\times \Ludud{}_h,\quad c((\pvec_h,\phi_h),(\qvec_h,\psi_h)) = {(S_{f},\psi_h)_{0,\Omega}}.
\end{array}\right.
\end{equation}
Following the definition in \cite[Corollary 26.15]{ern2021finiteII}, we assume that $(\Qvec_h)_h$, resp. $(L_h)_h$ have the {\em approximability property {in $\Omega$}} in the sense that 
\begin{equation}\label{eq:approx-pro}
\begin{array}{l}
\ds\forall \qvec\in {\Qvec(\Omega)},\ \lim_{h\to0}\left(\inf_{\qvec_h\in\Qvec_h}\|\qvec-\qvec_h\|_{{\Qvec(\Omega)}}\right)=0, \cr
\ds\forall \psi\in L^2(\Omega),\ \lim_{h\to0}\left(\inf_{\psi_h\in L_h}\|\psi-\psi_h\|_{0,\Omega}\right)=0.
\end{array}
\end{equation}
We also impose that the space $L_h^0$ of piecewise constant fields on the {mesh} is included in  $L_h$, and that $\dive\Qvec_h\subset L_h$. 
We finally define:
\begin{equation}\nonumber
\Xcal_h = \left\{\,\xi_h=(\qvec_h,\psi_h)\in\Qvec_h\times L_h\right\}\,,\mbox{ endowed with }\|\cdot\|_{\Xcal}\,.
\end{equation}
\begin{remark} At some point, the discrete spaces are considered locally, i.e. restricted to {a single mesh element}. So, one introduces the local spaces $\Qvec_h(K)$, $L_h(K)$, $\Xcal_h(K)$ for every $K\in \Tcal_h$.
\end{remark}
Provided the above conditions are fulfilled, one may derive a uniform discrete inf-sup condition under the same assumptions as in theorem~\ref{th:VF-1-PC},
{We proceed by using the equivalent notion of uniform $\Tcoer$-coercivity, cf. section 1.3.2 in \cite{Ciar25}, i.e. one has to prove that
\[ \begin{array}{l}
\exists\alpha^\star,\beta^\star,h_0>0,\ \forall h\in(0,h_0],\ \exists \Tcoer_h\in{\cal L}(\Xcalmg_h) \mbox{ such that }\forall (\pvec_h,\phi_h)\in\Xcalmg_h,\cr
\|\Tcoer_h(\pvec_h,\phi_h)\|_{\Xcalmg} \le \beta^\star \|(\pvec_h,\phi_h)\|_{\Xcalmg},\mbox{ and } 
 \ c((\pvec_h,\phi_h),\Tcoer_h(\pvec_h,\phi_h)) \ge \alpha^\star \|(\pvec_h,\phi_h)\|_{\Xcalmg}^2.
\end{array} \]
}
\begin{theorem}\label{th:disc+udisc-PC}
Let $\T_{o}\in\Pcal{\mathbb W}^{1,\infty}(\Omega)$ and $\T_{e}\in\Pcal{\mathbb W}^{1,\infty}(\Omega)$ satisfy \eqref{Pos-SPN}. 
 Assume that $(\Qvec_h)_h$, $(L_h)_h$ fulfill \eqref{eq:approx-pro}, $L_h^0\subset L_h$ and $\dive\Qvec_h\subset L_h$ for all $h$. Then the bilinear form $c$ {is uniformly $\Tcoer$-coercive}.
\end{theorem}
{As is classical when one uses the $\Tcoer$-coercivity theory, we retrace the steps of the proof of Theorem~\ref{th:VF-1-PC}, adding indices $_h$ along the process}.
\begin{proof}
{Given $h$, we define $\Tcoer_h\in{\cal L}(\Xcalmg_h)$ by} $$(\pvec_h,\phi_h) \mapsto(-\pvec_h,\frac{1}{2}(\phi_h + \T_{e,\text{inv},h}^{T}\H^T\dive\pvec_h)),$$ where $\T_{e,\text{inv},h}$ is the matrix of the projection of all the {entries} of $\T_{e}^{-1}$ onto $L^0_h$.  Indeed, since the matrix $\T_{e,inv,h}$ is piecewise constant, {for all $\pvec_h\in\Qududvec{}_h$, it holds that $\T_{e,inv,h}^T\H^T\dive\pvec_h\in\Ludud{}_h$}. Moreover, applying~\cite[Theorem 18.18]{ern2021finiteI} to each {entry of $\T_{e}^{-1}$} yields 
\begin{equation}
{\exists C_e>0,\ \forall h},\ \forall \psi\in \Ludud^2(\Omega),\quad \|(\T_{e}^{-1}-\T_{e,\text{inv},h})\psi\|_{0,\Omega}\le {C_e}h\|\psi\|_{0,\Omega}. \label{eq:approx_Te}
\end{equation}
Using the triangular inequality at the second line and estimates~\eqref{eq:pos_SPN_mixed} and~\eqref{eq:approx_Te} at the last line, we have
\begin{align}
\|{\Tcoer_h(\pvec_h,\phi_h)}\|_{\Xcalmg}^2& = \|\pvec_h\|_{\Qududvec(\Omega)}^2 +\frac{1}{4}\|\phi_h + \T_{e,\text{inv},h}^{T}\H^T\dive\pvec_h\|_{0,\Omega}^2\nonumber\\
&\leq\|\pvec_h\|_{\Qududvec(\Omega)}^2 +\frac{1}{4}(\|\phi_h\|_{0,\Omega} + \|\T_{e}^{-T}\H^T\dive\pvec_h\|_{0,\Omega} \nonumber\\
&\qquad\qquad\qquad\qquad +\|(\T_{e,\text{inv},h}^{T}-\T_{e}^{-T})\H^T\dive\pvec_h\|_{0,\Omega})^2\nonumber\\
&\leq\|\pvec_h\|_{\Qududvec(\Omega)}^2 + \frac{3}{4}(\|\phi_h\|_{0,\Omega}^2 + \|\T_{e}^{-T}\H^T\dive\pvec_h\|_{0,\Omega}^2 \nonumber\\
&\qquad\qquad\qquad\qquad +\|(\T_{e,\text{inv},h}^{T}-\T_{e}^{-T})\H^T\dive\pvec_h\|_{0,\Omega}^2)\nonumber\\
&\leq (1+\frac{3}{4}\{(\T_e^{-1})^*(\H^T)^*\}^2+\{{C_e}h(\H^T)^*\}^2)\|\pvec_h,\phi_h\|_{\Xcalmg}^2.\nonumber
\end{align}
{Hence, the mappings $(\Tcoer_h)_h$ are uniformly bounded.} \\
According to the definition of the bilinear form $c$, we have
\begin{align}
&\qquad c((\pvec_h,\phi_h),\Tcoer_h(\pvec_h,\phi_h))\nonumber\\
& = (\T_o\,\pvec_h,\pvec_h)_{0,\Omega}+\frac{1}{2}(\T_{e,\text{inv},h}^{T}\H^T\dive\pvec_h,\H^T\dive\pvec_h)_{0,\Omega}  \nonumber\\
&\quad
+ (\tilde{\Gamma}_e(\pvec_h\cdot\nvec), (\pvec_h\cdot\nvec))_{0,\pa\Omega}  +\frac{1}{2}(\T_e\,\phi_h,\phi_h )_{0,\Omega}\nonumber\\
&\quad + \frac{1}{2}((\T_{e,\text{inv},h}\T_e-\mathtt{I})\phi_h,\H^T\dive\pvec_h)_{0,\Omega},\nonumber\\
&\geq (\T_o)_*\|\pvec_h\|_{0,\Omega}^2  + \{\frac{1}{2}(\T_e^{-1})_*(\H^T)_*-{C_e(\T_e)^*}h\}\|\dive\pvec_h\|_{0,\Omega}^2 \nonumber \\
&\quad + (\tilde{\Gamma}_e)_* \|(\pvec_h\cdot\nvec)\|_{0,\partial\Omega}^2+ \frac{1}{2}(\T_e)_*\|\phi_h\|_{0,\Omega}^2 - {C_e(\T_e)^*}h\|\phi_h\|_{0,\Omega}\|\dive\pvec_h\|_{0,\Omega},\nonumber\\
&\geq \left\{\min\{(\T_o)_*,\frac{1}{2}(\H^T)_*(\T_e^{-1})_*, (\tilde{\Gamma}_e)_*, \frac{1}{2}(\T_e)_*)\}-{C_e(\T_e)^*}h\right\} \|\zeta\|_{\Xcalmg}^2.
\end{align}
The limit (as $h$ goes to $0$) of the coefficient between brackets is strictly positive, so the claim is proven.
\end{proof}
The classical {\em a priori} error analysis follows {(Céa's lemma)}. Let $\zeta_h=(\pvec_h,\phi_h)$ be the solution to \eqref{eq:VF-1h-PC}.
\begin{corollary}\label{cor:cv-PC}
{Under the assumptions of Theorem~\ref{th:disc+udisc-PC}}, there holds:
\begin{equation}\label{eq:error-estimate}
\exists\,C>0,\quad\forall h,\quad
\|\zeta-\zeta_h\|_{\Xcalmg} \le C\,\inf_{\xi_h\in\xvec_h}\|\zeta-\xi_h\|_{\Xcalmg}.
\end{equation}
\end{corollary}
Explicit {\em a priori} error estimates may be derived, see eg. \cite{CGJK18,Gerv26}. \\
In this paper, we focus on the Raviart-Thomas-N\'ed\'elec (RTN) Finite Element~\cite{RaTh77,Nedelec1980}. \\
{ For {\em simplicial meshes}}, that is meshes made of simplices, the finite element spaces RTN${}_k$ can be described as follows, where $k\ge0$ is the order of the discretization for  the scalar fields of $L_h$, see eg. \cite{BoBF13}. \\ 
The boundary of a simplex $K\in\Tcal_h$ is made of the union of $(d-1)$-simplices, called {\em facets} from now on, and denoted by $(F_e^K)_{1\le e\le d+1}$. We let $\mathbb{P}_k(K)$ be the space of polynomials of maximal degree $k$ on $K$, resp. $\mathbb{P}_k(F_e^K)$ the space of polynomials of maximal degree $k$ on $F_e^K$. The definition is 
\begin{eqnarray*} 
 &&\hskip -8truemm\mathrm{RTN}_k(K) = \{\qvec\in \Lvec^2(K) \,|\,\exists \avec\in (\mathbb{P}_k(K))^d,\,\exists b\in \mathbb{P}_k(K),\ \forall\xvec\in K,\ \qvec(\xvec) = \avec + b\xvec\nonumber \}.
\end{eqnarray*}
{ Observe that for all $\qvec\in \mathrm{RTN}_k(K)$, for all $e\in\{ 1,\cdots,d+1\},\  (\qvec\cdot\nvec)_{|F_e^K}\in \mathbb{P}_k(F_e^K)$.}
The definitions of the finite element spaces RTN${}_k$ are then 
\begin{align*}
&\Qvec_{h}= \{\qvec_h\in {\Qvec(\Omega)} \ |\ \forall K\in\Tcal_h,\ \qvec_h{}_{|K} \in  \mathrm{RTN}_k(K)\},\ \\
&L_h= \{\psi_h\in L^2(\Omega) \ |\ \forall K\in\Tcal_h,\ \psi_h{}_{|K} \in  \mathbb{P}_k(K)\}.
\end{align*}
 For {\em rectangular {or Cartesian} meshes}, a description of the Raviart-Thomas-N\'ed\'elec (RTN) finite element spaces can be found for instance in Section 4.2 of \cite{JaCi13}. {We consider those meshes explicitly for the numerical examples, see Section~\ref{sec:numerical_results}}.

\section{{\em A posteriori} studies for a mixed Finite element discretization}\label{section:error-estimator}
To develop the study of {\em a posteriori} estimates, we use the {so-called} reconstruction of the discrete solution { $\zeta_h$. In what follows, we denote by $\tilde{\zeta}_h:=\tilde{\zeta}_h(\zeta_h)$ a reconstruction, and by $\eta:=\eta(\tilde{\zeta}_h)$ an estimator. Classically, our aim is to obtain {\em reliable} and {\em efficient} estimators for the reconstructed error $\zeta-\tilde{\zeta}_h$, meaning that:
\[ \begin{array}{ll}
\|\zeta-\tilde{\zeta}_h\| \le \mathtt{C}\,\eta & \mbox{(reliability)} \cr
\eta \le \mathtt{c}\,\|\zeta-\tilde{\zeta}_h\| & \mbox{(efficiency)} \end{array} \]
where $\mathtt{C}$ and $\mathtt{c}$ are generic constants, and $\|\cdot\|$ is some norm to measure the error.} To that aim, the original space of solutions $\Hudud^1(\Omega)$ (see \eqref{eq:diff_primal}), {is assumed from now on to be} the {\em default} space of (scalar) reconstructed fields, {and we let $V = H^1(\Omega)$}. We also introduce the {\em broken spaces}
\begin{align*}
&H^1(\Tcal_h) = \{\psi\in L^2(\Omega) \ |\  \psi\in H^1(K) , \forall K \in \Tcal_h \},\\
&{\Qvec(\Tcal_h) = \{\qvec\in \Lvec^2(\Omega) \ |\  \qvec\in \Hvec(\dive;K) , \forall K \in \Tcal_h \text{ and } (\qvec\cdot\nvec)_{|_F} \in L^2(F) , \forall F \in\mathcal{F}^e_h \}}. 
\end{align*}
Following the approach in~\cite{CDM23}, the reconstruction $\tilde{\zeta}_h=(\tilde{\pvec}_h,\tilde{\phi}_h)$ is defined as 
\begin{align*}
&\tilde{\pvec}_h= \pvec_h \in \Qududvec{}_h\subset \Qududvec(\Omega),\\
&{\tilde{\phi}_h \in \Vudud}.
\end{align*}
In {Section~\ref{sec:reconstructions}}, {we recall some reconstruction approaches for RTN finite element spaces}. Section~\ref{sec:error_estimates} is devoted to the derivation of {\em a posteriori} estimates.

\subsection{{Reconstruction of the discrete solution}}\label{sec:reconstructions}
In this section, we present some approaches to devise a reconstruction of the discrete solution $(\pvec_h,\phi_h)$, here obtained with the RTN${}_k$ finite element discretization, for $k\ge0$. {Below, the novelty consists in taking into account the Robin boundary condition}. \\
 For illustrative purposes, we consider simplicial meshes (see Remark~\ref{remark_reconstruction_rectangle}).
We denote by $\mathbb{P}_k(\mathcal{T}_h)$ the space of piecewise polynomials of maximal degree $k$ on each {(closed)} simplex $K\in \mathcal{T}_h$. We let $\mathcal{V}_h^k$ be the set of interpolation points (or nodes) where the degrees of freedom of the $V$-conforming Lagrange Finite Element space of order $k$ are defined. And, for a node $a\in\mathcal{V}_h^k$, we denote by $\mathcal{T}_a$ the set of simplices $K$ such that $a\in K$. \\
The definition of the {(original)} {Oswald interpolation operator~\cite{oswald1993bpx}}
{$\mathcal{I}_{\text{Os}} : \mathbb{P}_k(\mathcal{T}_h) \to \mathbb{P}_{k}(\mathcal{T}_h) \cap V$} is
\begin{equation*}
{\forall \phi_h\in\mathbb{P}_k(\mathcal{T}_h),\ \forall a\in\mathcal{V}_h^k},\quad \mathcal{I}_{\text{Os}}(\phi_h)(a)= \frac{1}{|\mathcal{T}_a|} \displaystyle \sum_{K\in \mathcal{T}_a} \phi_h{}_{|K}(a).
\end{equation*}

\begin{remark}\label{remark_reconstruction_rectangle}
The results presented in this section can be extended to the case of rectangular {or cuboid} meshes~\cite{Vohralik2010}. 
\end{remark}

\subsubsection{Averaging operator}
\label{sec:average_reconstruction}
We introduce the averaging operator of the neutron flux 
$\mathcal{I}_{av} : \Pudud_k(\mathcal{T}_h) \to \Pudud_{k+1}(\mathcal{T}_h) \cap V$ such that $\forall \phi_h\in\Pudud_k(\mathcal{T}_h)$,
\begin{equation*}
{ \forall a\in\mathcal{V}_h^{k+1}},\quad \mathcal{I}_{av}(\phi_h)(a)=
 \left\{\begin{aligned}
 &{\frac{1}{|\mathcal{T}_a|} \displaystyle \sum_{K\in \mathcal{T}_a} ({\Gamma_e^{-1}\H^T(\pvec_h\cdot\nvec)}){}_{|K}(a) \quad \text{ if }a\in\partial\Omega,}\\ 
 &\frac{1}{|\mathcal{T}_a|} \displaystyle \sum_{K\in \mathcal{T}_a} \phi_h{}_{|K}(a)\quad \text{ otherwise.}
 \end{aligned}\right.
\end{equation*}
\begin{remark}
We note that, for $a\in\mathcal{V}_h^{k+1}$, $\Gamma_e^{-1}\H^T(\pvec_h\cdot\nvec)(a)$ can be multi-valued. On the other hand, for all $K\in \mathcal{T}_h$, $\Gamma_e^{-1}\H^T(\pvec_h\cdot\nvec)$ is single-valued over $K$. Hence, the value $(\Gamma_e^{-1}\H^T(\pvec_h\cdot\nvec)){}_{|K}(a)$ is well-defined. Likewise, the equality $\H^T(\pvec_h\cdot\nvec)-\Gamma_e\mathcal{I}_{av}(\phi_h)=0$ does not hold over $\partial\Omega$. Indeed, we recall that, $\H^T(\pvec_h\cdot\nvec)_{|\partial\Omega}$ is only piecewise smooth (and continuous if and only if it is equal to a constant), while $(\mathcal{I}_{av}(\phi_h))_{|\partial\Omega}$ is automatically continuous.
\end{remark}
The {\em average reconstruction} is then
\begin{equation}
\tilde{\zeta}_{av,h}=(\pvec_h,\mathcal{I}_{av}(\phi_h)).
\label{eq:reconstruction_average}
\end{equation}
{Importantly,} the definition of the operator on the interpolation points located at the boundary is {driven} by the a posteriori analysis (cf.~\eqref{eq:def_bc_estimator}) presented in Section~\ref{sec:error_estimates}.

\subsubsection{Post-processing approach}
We {outline next} the approach proposed in~\cite{Arbogast1995}, valid for $k\ge0$. It is shown there that the solution to~\eqref{eq:VF-1h-PC}, $\zeta_h=(\mathbf{p}_{h},\phi_h) \in \Xcalmg_h$, is also equal to the first argument of the solution of a hybrid formulation, where the constraint on the continuity of the normal trace of $\pvec_h$ is relaxed. Let
\[ \Lambda_h=\left\{\lambda_h\in L^2(\mathcal{F}^i_h)\ |\ \exists \qvec_h \in \Qvec_h,\, \lambda_h{}_{|F}=\qvec_h\cdot \nvec_{|F},\, \forall F\in \mathcal{F}^i_h \right\}, \] be the space of the Lagrange multipliers and let $\tilde{\Xcal}_h=\Pi_{K\in\Tcal_h}\Xcal_h(K)$ be the unconstrained approximation space {with the RTN${}_k$ local finite element spaces. By definition, $\Xcal_h$ is a strict subset of $\tilde{\Xcal}_h$}. \\
{The hybrid formulation is:}
\begin{equation}\label{eq:hybrid_formulation}
\left\{\begin{array}{l}
\mbox{Find $(\zeta_h,\lambda_h)\in \tilde{\Xcalmg}_h\times\Lambdaudud_h$ such that }\forall (\xi_h,\mu_h)\in \tilde{\Xcalmg}_h\times\Lambdaudud_h,\cr
\ds c(\zeta_h,\xi_h)-\sum_{F\in \mathcal{F}^i_h}\int_F \lambda_h[\qvec_h\cdot \nvec]+\sum_{F\in \mathcal{F}^i_h}\int_F \mu_h[\pvec_h\cdot \nvec]={(S_{f},\psi_h)_{0,\Omega}}.
\end{array}\right.
\end{equation}

Let $\Pi_{M_h}:\tilde{\Xcalmg}_h\times\Lambdaudud_h \to {\Mudud_h}$ be the projection onto an appropriate space {$\Mudud_h$ (we refer to~\cite{Arbogast1995,CiDoGeMa25} for the definition of $M_h$)}
such that, {given $(\zeta_h,\lambda_h)\in \tilde{\Xcalmg}_h\times\Lambdaudud_h$, its projection $\widehat{\phi}_h=\Pi_{M_h}(\zeta_h,\lambda_h)$ is governed by}
 \[ 
 \forall (\psi_h,\mu_h)\in {\Ludud_h}\times \Lambdaudud_h,\quad ({\widehat{\phi}_h},\psi_h)_{0,\Omega} + \sum_{F\in \mathcal{F}^i_h}\int_F { \widehat{\phi}_h}\mu_h = (\phi_h,\psi_h)_{0,\Omega} + \sum_{F\in \mathcal{F}^i_h}\int_F \lambda_h\mu_h.
 \]

{Taking into account the Robin boundary condition}, the RTN post-processing is defined {here} by $ \mathcal{I}^2_{\text{RTN}} : \tilde{\Xcalmg}_h\times\Lambdaudud_h \to \Pudud_{k+2}(\mathcal{T}_h) \cap {\Vudud}$ such that $\forall (\zeta_h,\lambda_h)\in\tilde{\Xcalmg}_h\times\Lambdaudud_h,$
\begin{equation*}
{ \forall a\in\mathcal{V}_h^{k+2}},\quad \mathcal{I}^2_{\text{RTN}}{(\zeta_h,\lambda_h)}(a)=  \left\{\begin{aligned}
&{\frac{1}{|\mathcal{T}_a|} \displaystyle \sum_{K\in \mathcal{T}_a} ({\Gamma_e^{-1}\H^T(\pvec_h\cdot\nvec)})_{|K}(a) \quad  \text{ if }a\in\partial\Omega,}\\
&\frac{1}{|\mathcal{T}_a|} \displaystyle \sum_{K\in \mathcal{T}_a} { (\Pi_{M_h}(\zeta_h,\lambda_h))}{}_{|K}(a) \quad  \text{otherwise.}
\end{aligned}\right.
\end{equation*}
The {\em reconstruction} associated to the RTN post-processing is 

\begin{equation}
\tilde{\zeta}_{\text{RTN},h}=(\pvec_h, \mathcal{I}^2_{\text{RTN}}{(\zeta_h,\lambda_h)}).
\label{eq:reconstruction_RTN_post-processing}
\end{equation}

\subsection{{\em A posteriori} error estimates}
\label{sec:error_estimates}
We now detail the derivation of {\em a posteriori} estimates. We define 
\begin{align*}
&d_S(\zeta,\xi) = { (\T_o\,\pvec,\qvec)_{0,\Omega}+(\T_e\phi,\psi)_{0,\Omega}}{+ (\tilde{\Gamma}_e(\pvec\cdot\nvec), (\qvec\cdot\nvec))_{0,\pa\Omega}},\\ 
&d(\zeta,\xi)= d_S(\zeta,\xi){ +(\psi,\H^T\dive\pvec)_{0,\Omega}-(\phi,\H^T\dive\qvec)_{0,\Omega}}= c(\zeta,(-\mathbf{q},\psi)). \\
\end{align*}
{It is understood that} the definition is extended to piecewise smooth fields on $\Tcal_h$ by replacing $\displaystyle\int_{\Omega}$ by $\displaystyle \sum_{K\in \Tcal_h}\int_K$. \\
Given $K\in\Tcal_h$, we also define $\pi_0^K$ the $L^2(K)$-orthogonal projection on the space $L^0_h(K)$, $\delta_{e,o}$ the diagonal part of the matrix $\T_{e,o}$, and
\[ \delta^{max}_{e,K} = \max_{g\in\Ical_G, i\in\Ical_e}\sup_{K}(((\T_{e})_{g,g})_{i,i}),\quad
\delta^{min}_{e,K} = \min_{g\in\Ical_G, i\in\Ical_e}\inf_{K}(((\T_{e})_{g,g})_{i,i}) ,
\]
\[ \delta^{max}_{o,K} = \max_{g\in\Ical_G, i\in\Ical_o}\sup_{K}(((\T_{o})_{g,g})_{i,i}),\quad
\delta^{min}_{o,K} = \min_{g\in\Ical_G, i\in\Ical_o}\inf_{K}(((\T_{o})_{g,g})_{i,i}) ,
\]
\[(\tilde{\Gamma}_e)^* = \sup_{X\in\underline{\underline{\R}}\setminus\{0\}}\frac{X^T\tilde{\Gamma}_eX}{\|X\|^2}.\]
In order to state the estimates, { at some point} we will use the following assumptions.
\begin{assumption}\label{assumption:locality_polynomial}
The coefficients of $\T_{o}$ and $\T_{e}$ are piecewise polynomials
 on $\Tcal_h$, and $S_f\in {\Ludud_h}$. {In addition, we suppose that $\tilde{\phi}_h$ is piecewise polynomial on $\Tcal_h$.}
\end{assumption}

In~\cite{CDM23}, some of the co-authors proposed two alternatives: for the first one they measure the error with respect to the {basic $\Lvec^2(\Omega) \times L^2(\Omega)$ norm, while for the second one they use the strenghtened ${\Hvec(\dive,\Tcal_h)} \times L^2(\Omega)$ norm. We are dealing with a Robin boundary condition, so one has to incorporate a measure of the normal trace of the vector-valued fields. Since existence of the normal trace is guaranteed for elements of $\Hvec(\dive,\cdot)$ this indicates that an appropriate norm should be based on the strengthened norm}. {For this reason, we introduce} the norm $\|\cdot\|_S$ on $\Xcalmg$ where, for all $\zeta\in \Xcalmg$,
\begin{align} 
\|\zeta\|_S^2
&= (\delta_o\pvec,\pvec)_{0,\Omega}
+(\delta_e\,\phi,\phi)_{0,\Omega} \nonumber\\
&\quad 
+{\sum_{K\in \Tcal_h} \delta^{max}_{o,K}h_K^2 \| \dive \pvec\|_{0,K}^2}
{+ \sum_{F\in\mathcal{F}^e_h} \delta^{max}_{o,K_F}h_{\perp F} \|\tilde{\Gamma}_e^{1/2}(\pvec\cdot\nvec)\|_{0,F}^2},\label{strenghtened-norm}
\end{align}
where $h_{\perp F}$ the length of the altitude associated to the face $F\in\mathcal{F}^e_h$ in the mesh element $K_F$ such that $F\subset\partial K$.
Observe that the norm $\|\cdot\|_S$ measures elements of $\Xcalmg$ in a weighted $\Qududvec(\Tcal_h)\times \Ludud^2(\Omega)$ norm, {similarly to~\cite[\S8]{Giret2018})}. \\
{For $K\in\Tcal_h$, we introduce} $$N(K)=\{K'\in\Tcal_h\ |\ \text{dim}_{H}({\partial K'\cap \partial K})=d-1\},$$ where dim$_H$ is the Hausdorff dimension, and $$\Xcalmg_K  = \left\{\zeta=(\pvec,\phi)\in \Xcalmg\ |\ \text{Supp}(\phi) \subset K, \text{Supp} (\pvec) \subset N(K)\right\}.$$ Then one can define the following $\Xcalmg_K$-local norm, for all $\zeta\in \Xcalmg$,
\begin{equation} \label{local-norm}
|\zeta|_{+,K}=\sup_{\xi\in \Xcalmg_K, \|\xi\|_S\leq 1}d(\zeta,\xi).
\end{equation}

\begin{lemma}
\label{lemma:leandre_reconstruction}
Let $\zeta$ and $\zeta_h=(\pvec_h,\phi_h)$ be respectively the solution to~\eqref{eq:VF-3-PC} and~\eqref{eq:VF-1h-PC}. 
Let $\tilde{\zeta}_h=(\pvec_h,\tilde{\phi}_h){\in\Qududvec{}_h\times {\Vudud}}$ be a reconstruction of $\zeta_h$. We have for all $\xi{=(\qvec,\psi)}\in \Xcalmg$,
\begin{align}
d(\zeta-\tilde{\zeta}_h,\xi)
&= (S_f - \H^T\dive \pvec_h - \T_e\tilde{\phi}_h, \psi)_{0,\Omega}- (\T_o\pvec_h +\H\grad\tilde{\phi}_h , \mathbf{q})_{0,\Omega} \nonumber\\
&\quad {+(\H\tilde{\phi}_h- \tilde{\Gamma}_e(\pvec_h\cdot\nvec), (\qvec\cdot\nvec))_{0,\pa\Omega}}.
\label{eq:recons_leandre}
\end{align}
\end{lemma}
\begin{proof}
Let $\xi$ be in $\Xcalmg$. According to~\eqref{eq:VF-3-PC}, we have
\begin{align*}
d(\zeta-\tilde{\zeta}_h,\xi)
&= (S_f - \H^T\dive \pvec_h - \T_e\tilde{\phi}_h, \psi)_{0,\Omega}
- (\T_o\pvec_h  , \mathbf{q})_{0,\Omega}+ (\tilde{\phi}_h, \H^T\dive \mathbf{q})_{0,\Omega}
\nonumber\\
&\quad {- (\tilde{\Gamma_e}(\pvec_h\cdot\nvec), (\qvec\cdot\nvec))_{0,\pa\Omega}}.
\end{align*}
Using $\tilde{\phi}_h\in{\Vudud}$, we can integrate by part the third integral {to recover (\ref{eq:recons_leandre})}.
\end{proof}

\begin{definition}\label{definition:estimators}
{Let $\zeta_h=(\pvec_h,\phi_h)$ be the solution to~\eqref{eq:VF-1h-PC}}. Let $\tilde{\zeta}_h=(\pvec_h,\tilde{\phi}_h) \in{\Qududvec{}_h}\times {\Vudud}$ be a reconstruction of $\zeta_h$. For any $K\in \Tcal_h$, 
we define 
 \begin{eqnarray}
&{\mbox{the {\em residual estimator}:}}\hskip 3truemm
{\eta_{r,K}}=\|\delta_e^{-1/2} (S_f-\H^T\dive \pvec_h -\T_e\tilde{\phi}_h)\|_{0,K}\,,\label{eq:def_residual_estimator} \\
&{\mbox{resp., the {\em flux estimator}:}} \hskip 3truemm
{\eta_{f,K}} = \|\delta_o^{-1/2}(\T_o\pvec_h+\H\grad \tilde{\phi}_h)\|_{0,K},\label{eq:def_flux_estimator}
\end{eqnarray}
and, for any $F\in {\mathcal{F}_h^e}$, we define the {\em Robin boundary condition estimator}:
\begin{align}\label{eq:def_bc_estimator}
{}{\eta_{bc,F}} = (\delta^{max}_{o,K_F}h_{\perp F})^{-1/2}\|\tilde{\Gamma}_e^{-1/2}(\H\tilde{\phi}_h-\tilde{\Gamma}_e(\pvec_h\cdot\nvec))\|_{0,F}.
\end{align}
\end{definition}
\begin{theorem}[reliability]\label{theorem:norm+K}
{Let $\zeta$ be the solution to~\eqref{eq:VF-3-PC}. With the same notation as in definition~\ref{definition:estimators}}, one has the estimate
\begin{align}
&|\zeta-\tilde{\zeta}_h|_{+,K}\leq \left({\eta}^2_{r,K} +\sum_{K'\in N(K)}\eta^2_{f,K'}{+\sum_{F\in {\mathcal{F}_h^e}\cap \pa K}\eta_{bc,F}^2}\right)^{1/2}.\label{eq:rt_ineq_1}
\end{align}
\end{theorem}
\begin{proof}
According to Lemma~\ref{lemma:leandre_reconstruction}, we have
\begin{align*}
d(\zeta-\tilde{\zeta}_h,\xi)
&= (S_f - \H^T\dive \pvec_h - \T_e\tilde{\phi}_h ,\psi)_{0,\Omega} - (\T_o\pvec_h +\H\grad\tilde{\phi}_h , \mathbf{q})_{0,\Omega}\\
&\quad {+(\H\tilde{\phi}_h- \tilde{\Gamma}_e(\pvec_h\cdot\nvec), (\qvec\cdot\nvec))_{0,\pa\Omega}}.
\end{align*}
Let  $K\in\Tcal_h$ and $\xi=(\qvec,\psi)\in \Xcalmg$ be such that $\text{Supp}(\psi) \subset K, \text{ Supp} (\qvec) \subset N(K)$. Applying Cauchy-Schwarz inequalities {successively in $\Ludud^2(K)$, $\Ludud^2(K')$ for $K'\in N(K)$, $\Ludud^2(F)$ for $F\in {\mathcal{F}_h^e}\cap \pa K$ and {finally in $\mathbb{R}^{m}$ (for ad hoc $m$)}, we get
\begin{align*}
&\qquad d(\zeta-\tilde{\zeta}_h,\xi)\\
&\leq {\eta}_{r,K}\|\delta_e^{1/2}\psi\|_{0,K}  + \sum_{K'\in N(K)} \eta_{f,K'}\|\delta_o^{1/2}\mathbf{q}\|_{0,K'}{+\sum_{F\in {\mathcal{F}_h^e}\cap \pa K}{\eta_{bc,F}}\|\tilde{\Gamma}_e^{1/2}(\qvec\cdot\nvec)\|_{0,F}}\nonumber\\
&\leq  \left({\eta}_{r,K}^2+\sum_{K'\in N(K)} \eta_{f,K'}^2{+\sum_{F\in {\mathcal{F}_h^e}\cap \pa K}\eta_{bc,F}^2}\right)^{1/2}\\
&\qquad \times  { \scriptstyle \left(\|\delta_e^{1/2}\psi\|_{0,K}^2  + \sum_{K'\in N(K)}\|\delta_o^{1/2}\mathbf{q}\|_{0,K'}^2{+\sum_{F\in {\mathcal{F}_h^e}\cap \pa K} \delta^{max}_{o,K_F}h_{\perp F}\|\tilde{\Gamma}_e^{1/2}(\qvec\cdot\nvec)\|_{0,F}^2}\right)^{1/2}.}\end{align*}
We infer~\eqref{eq:rt_ineq_1} from the definition  of the $|\cdot |_{+,K}$ norm \eqref{local-norm}}.
\end{proof}

\begin{remark}\label{rmk:diffusion_multigroup}
The reliability estimate for the multigroup neutron diffusion equation~\cite{CDM25} may be explicitly stated since it corresponds to the specific case where $\Nhat=1$, $\Gamma_e=\frac{1}{2}$ and $\T_o^g=\frac{1}{D^g}$ for all $g\in\Ical_G$ with $D^g$ the scalar-valued diffusion coefficient of the energy group $g$. Notice that ${\T_o\in\R^{G\times G}}$ is a diagonal matrix in this case. Denoting ${\D= (\T_o)^{-1}\in\R^{G\times G}}$ the diffusion matrix then, for any $K\in \Tcal_h$ and any $F\in {\mathcal{F}_h^e}$, the estimators write 
\begin{align*}
&{\eta_{r,K}}=\|\delta_e^{-1/2} (S_f-\dive \pvec_h -\T_e\tilde{\phi}_h)\|_{0,K}\,, \\
&{\eta_{f,K}} = \|\D^{1/2}(\D^{-1}\pvec_h+\grad \tilde{\phi}_h)\|_{0,K}, \\
&{\eta_{bc,F}} = (\delta^{max}_{o,K_F}h_{\perp F})^{-1/2} \|\frac{1}{\sqrt{2}}(\tilde{\phi}_h-2(\pvec_h\cdot\nvec))\|_{0,F}.
\end{align*}
\end{remark}

%
{
\begin{theorem}
[efficiency]
\label{thm_apost_error_est}
Let Assumption~\ref{assumption:locality_polynomial} be fulfilled. 
For $K \in \Tcal_h$, let ${\eta}_{r,K}$ and ${\eta}_{f,K}$ be the residual and flux estimators respectively given by~\eqref{eq:def_residual_estimator}, and~\eqref{eq:def_flux_estimator}. The following estimates hold true
\begin{align}
{\eta}_{r,K}&\leq \mathtt{c}\,{ \left(\frac{\delta^{max}_{e,K}}{\delta^{min}_{e,K}}\right)^{1/2}}\,|\zeta-\tilde{\zeta}_h|_{+,K}
, \label{eq:norm+_local_efficiency_eta_r}\\
{{\eta}_{f,K}}&{\leq { \mathtt{C} \left(\frac{\delta^{max}_{o,K}}{\delta^{min}_{o,K}}\right)^{1/2}}\,|\zeta-\tilde{\zeta}_h|_{+,K}}
,\label{eq:norm+_local_efficiency_eta_f}
\end{align}
where {$\mathtt{c}$ and $\mathtt{C}$ are constants which depend} only on the polynomial degree of $S_f$, $\T_o$, $\T_e$  {and $\tilde{\phi}_h$}, $d$, and the shape-regularity parameter $\kappa_K$. \\
{For $F \in \mathcal{F}^e_h$, let ${\eta}_{bc,F}$ be the Robin boundary condition estimator given by~\eqref{eq:def_bc_estimator}. 
The following estimate holds true
\begin{align}
{{\eta}_{bc,F}}&{\leq \mathsf{C} \,|\zeta-\tilde{\zeta}_h|_{+,K_F}}
,\label{eq:norm+_local_efficiency_eta_bc}
\end{align}
where $h_{\perp F}$ is the size of the $F$-transverse part of the {mesh element $K_F$} containing $F$ in its facets, $\mathsf{C}$ is a constant which depends only on the polynomial degree of $S_f$,  $\T_e$  {and $\tilde{\phi}_h$}, $d$, and the shape-regularity parameter $\kappa_{K_F}$.
}
\end{theorem}
\begin{proof}
{ The first part of the proof is similar to that of~\cite[Theorem 5.7]{CDM23}. 
{
Let $\psi_K$ be the bubble function on $K$: {if $K$ is a simplex, it is}
given as the product of the $d+1$ linear functions that take the value 1 at one vertex of $K$
and vanish at the other vertices; {if $K$ is a rectangle or a cuboid}, it is given as the product of the $2d$ linear functions that take the value 1 on one face $F\subset \partial K$ and vanish on the opposite face.
Let $\psi_r = (S_f-\H^T\dive \pvec_h-\T_e\widetilde{\phi}_h)$.
Note that $\psi_r$ is a polynomial in $K$, because each term appearing in its definition is a polynomial (thanks to Assumption~\ref{assumption:locality_polynomial} for $S_f$, $\T_e$ {and $\widetilde{\phi}_h$}). Then the equivalence of norms on finite-dimensional spaces, { the definition of $\psi_K$} and the inverse inequality (cf., e.g., \cite[Theorem 3.2.6]{ciarlet2002finite}) respectively give 
\begin{align}
&{c_\psi}\|\psi_r\|_{0,K}^2\leq (\psi_r,\psi_K\psi_r)_{0,K},\label{eq:norm+_ineq_5}\\
&\|\psi_K \psi_r\|_{0,K}\leq \|\psi_r\|_{0,K},\label{eq:norm+_ineq_6}
\end{align}
with the constant ${c_\psi}$ depending only on the polynomial degree of $S_f$, { $\T_e$}  {and $\tilde{\phi}_h$}, $d$, and $\kappa_K$.
}\\
Now, let $\xi_{r,K}=(0,\psi_K\psi_r)$ in $K$, and $0$ elsewhere: as we observed previously, $\xi_{r,K}\in\Xcal$. Then we have, by the definition of the bilinear form $d$ and of $\zeta$
\begin{align*}
d(\zeta-\tilde{\zeta_h},\xi_{r,K})=(\psi_r,\psi_K\psi_r)_{0,K}.
\end{align*}
{ Since the support of $\xi_{r,K}$ is equal to $K$, one has actually $\xi_{r,K}\in\Xcal_K$. So,} by definition \eqref{local-norm} of the strengthened $|\cdot|_{+,K}$ norm},
\begin{align}
d(\zeta-\tilde{\zeta}_h,\xi_{r,K})&\leq |\zeta-\tilde{\zeta}_h|_{+,K} \|\xi_{r,K}\|_S\nonumber\\
&\leq|\zeta-\tilde{\zeta}_h|_{+,K}\|\delta_e^{1/2} \psi_K \psi_r\|_{0,K}.
\label{eq:norm+_ineq_4}
\end{align}
Combining~\eqref{eq:norm+_ineq_5},~\eqref{eq:norm+_ineq_6} and~\eqref{eq:norm+_ineq_4}, one comes to
\begin{align*}
{c_\psi}\|\psi_r\|_{0,K}^2\leq |\zeta-\tilde{\zeta}_h|_{+,K}\|\psi_r\|_{0,K}{(\delta^{max}_{e,K})^{1/2}}.
\end{align*}
Using the definition of $\eta_{r,K} $ by~\eqref{eq:def_residual_estimator} concludes
the proof of~\eqref{eq:norm+_local_efficiency_eta_r}: 
\begin{align*}
\eta_{r,K}{ \leq (\delta^{min}_{e,K})^{-1/2}\|\psi_r\|_{0,K}\leq  \frac{1}{{c_\psi}}\left(\frac{\delta^{max}_{e,K}}{\delta^{min}_{e,K}}\right)^{1/2} |\zeta-\tilde{\zeta}_h|_{+,K}}.
\end{align*}
We now proceed similarly for the second estimate.
Let us denote ${\qvec_f} = \T_o\pvec_h +\H\grad\tilde{\phi}_h$
on a given $K \in \Tcal_h$. Note that $\qvec_f$ is a polynomial in $K$ (thanks to Assumption~\ref{assumption:locality_polynomial} for $\T_o$ {and $\widetilde{\phi}_h$}). Then the
equivalence of norms on finite-dimensional spaces, { the definition of $\psi_K$} and the inverse inequality (cf., e.g., \cite[Theorem 3.2.6]{ciarlet2002finite})
give
\begin{align}
&{c_q}\|\qvec_f\|_{0,K}^2\leq (\qvec_f,\psi_K\qvec_f)_{0,K},\label{eq:norm+_eta_f_ineq_5}\\
&\|\psi_K \qvec_f\|_{0,K}\leq \|\qvec_f\|_{0,K},\label{eq:norm+_eta_f_ineq_6}\\
& \|\dive (\psi_K \qvec_f)\|_{0,K} \leq {C_d}\,h_K^{-1}\|\psi_K \qvec_f\|_{0,K},\label{eq:norm+_eta_f_ineq_7}
\end{align}
with the constants ${c_q}$ and ${C_d}$ depending only on the { polynomial degree of $\T_o$} {and $\tilde{\phi}_h$}, $d$, and
$\kappa_K$. \\
{ Let ${\xi_{f,K}}=(\psi_K\qvec_f,0)$ in $K$, and $0$ elsewhere. We observe that $\psi_K\qvec_f$ is smooth in $K$ (a closed subset of $\mathbb{R}^d$), and moreover that $(\psi_K\qvec_f)_{|\partial K}=0$ thanks to the definition of $\psi_K$. Hence, $\xi_{f,K}\in\Xcal_K$.} According to Lemma~\ref{lemma:leandre_reconstruction}
\begin{align*}
-d(\zeta-\tilde{\zeta_h},\xi_{f,K})&=(\T_o\pvec_h + \H\grad\tilde{\phi}_h , \psi_K\qvec_f)_{0,K} \\
&= (\qvec_f, \psi_K\qvec_f)_{0,K}. 
\end{align*}
By definition \eqref{local-norm} of the $|\cdot|_{+,K}$ norm, if now follows that
\begin{align}
&\quad 
-d(\zeta-\tilde{\zeta}_h,\xi_{f,K}) \nonumber\\
&\leq |\zeta-\tilde{\zeta}_h|_{+,K}\|\xi_{f,K}\|_S\nonumber\\
&\leq|\zeta-\tilde{\zeta}_h|_{+,K}
 \Big\{\|\delta_o^{1/2}(\psi_K \qvec_f)\|_{0,K}^2 +\delta^{max}_{o,K}h_K^2 \|\dive (\psi_K \qvec_f)\|_{0,K}^2 \Big\}^{1/2}\nonumber\\
&\leq|\zeta-\tilde{\zeta}_h|_{+,K}{ (\delta^{max}_{o,K})^{1/2}}
\Big\{  \|(\psi_K \qvec_f)\|_{0,K}^2 + { h_K^2} \| \dive (\psi_K \qvec_f)\|_{0,K}^2 \Big\}^{1/2}\nonumber\\
&\leq|\zeta-\tilde{\zeta}_h|_{+,K}
 { (\delta^{max}_{o,K})^{1/2}\{1+{C_d^2}\}^{1/2}}\|(\psi_K \qvec_f)\|_{0,K},\label{eq:norm+_eta_f_ineq_4}
\end{align}
{where we used the inverse inequality \eqref{eq:norm+_eta_f_ineq_7} to reach the last line.}
Combining~\eqref{eq:norm+_eta_f_ineq_5},~\eqref{eq:norm+_eta_f_ineq_6}
and~\eqref{eq:norm+_eta_f_ineq_4}, one comes to
\begin{align}\label{estimate_on_qf}
{c_q}\|\qvec_f{\|}_{0,K}^2\leq |\zeta-\tilde{\zeta}_h|_{+,K}\|\qvec_f\|_{0,K}{ (\delta^{max}_{o,K})^{1/2}\{1+{C_d^2}\}^{1/2}}.
\end{align}
Considering the definition of $\eta_{f,K} $ by~\eqref{eq:def_flux_estimator} concludes
the proof of the second estimate. \\
We finally prove the third estimate. Let $e_h=\H\tilde{\phi}_h- \tilde{\Gamma}_e(\pvec_h\cdot\nvec)\in L^2(\partial\Omega)$. 
Given $F \in \mathcal{F}_h^e$ associated to $K_F\in\Tcal_h$,  we build a "bubble" function $\psi_F$ vanishing not on the whole of $\pa {K_F}$, but only on $\overline{\partial {K_F} \setminus F}$. \\
{\small First, in the case where ${K_F}$ is a rectangle or a cuboid, the bubble function is defined as
\[ \psi_F=c_F\prod_{F'\subset\mathcal{F}_h\cap(\partial {K_F} \setminus F)} \lambda_{F'}, \] where 
$c_F>0$ and for all faces $F'\in \mathcal{F}_h\cap \partial {K_F}$, $ \lambda_{F'}(\xvec)=-\frac{(\xvec-\yvec_{F'})\cdot\nvec_{F'}}{h_{\perp F'}}$, with $\yvec_{F'}\in F'$ given. Note that for any face $F'\in \mathcal{F}_h\cap \partial {K_F}$, we have for all $\xvec\in \overline{{K_F}}$, $0\leq \lambda_{F'}(\xvec)\leq 1$, and for all $\xvec\in F'$, $\lambda_{F'}(\xvec)=0$.
For all $\xvec\in \overline{{K_F}}$, we also denote ${\qvec_{bc}}(\xvec)= e_h(\xvec+\lambda_F(\xvec)h_{\perp F}\nvec_F)\boldsymbol{\rho}_{{K_F},F}(\xvec) $ where $\boldsymbol{\rho}_{{K_F},F}(\xvec) = (1 - \lambda_F(\xvec))\nvec_F$.
 \\
Second, in the case where ${K_F}$ is a simplex, let us introduce $(\lambda_i)_{i=0,..,d}$ the normalized barycentric coordinates associated to ${K_F}$ and assume, without loss of generality, that the face $F$ is characterized by $\lambda_0=0$ and we index the other faces $(F_i)_{1\leq i\leq d}$ in $(\mathcal{F}_h\cap \partial {K_F})\setminus F$\,; finally $\xvec_i$ is the position of the vertex opposite to the face $F_i$. The bubble function is now defined as, 
\[ \psi_F=c_F\prod_{i=1}^d \lambda_{i}, \]  
with $c_F>0$. For all $\xvec\in \overline{{K_F}}$, we now denote
${\qvec_{bc}}(\xvec)= e_h(\sum_{i=1}^{d-1}\lambda_i(\xvec)\xvec_i +(\lambda_0(\xvec)+\lambda_d(\xvec))\xvec_d)\boldsymbol{\rho}_{{K_F},F}(\xvec) $ 
where $\boldsymbol{\rho}_{{K_F},F}(\xvec)  = (1-\lambda_0(\xvec) )\nvec_F$. We note that, for all $\xvec \in {K_F}$, the sum $\sum_{i=1}^{d-1}\lambda_i(\xvec)\xvec_i +(\lambda_0(\xvec)+\lambda_d(\xvec))\xvec_d$ actually belongs to $F$, with $(\lambda_1(\xvec),$ $\dots,$ $\lambda_{d-1}(\xvec),$ $\lambda_0(\xvec)+\lambda_d(\xvec))$ acting as barycentric coordinates there.} \\
In both cases, the constant $c_F$ is chosen such that $\|\psi_F\|_{L^\infty({K_F})}=1$.  Note that $\qvec_{bc}$ and $e_h$ are polynomials 
 respectively in ${K_F}$ and $F$, {moreover it holds that ${\qvec_{bc}}\cdot\nvec_{|F}(\xvec)= e_h(\xvec)$ for all $\xvec\in F$}. Then the
equivalence of norms on finite-dimensional spaces, the definition of $\psi_F$ and the inverse inequality (cf., e.g., \cite[Theorem 3.2.6]{ciarlet2002finite})
give
\begin{align}
&\|\psi_F \qvec_{bc}\|_{0,{K_F}}\leq \|\qvec_{bc}\|_{0,{K_F}},\label{eq:norm+_eta_f_ineq_6.1}\\
& \|\dive (\psi_F\qvec_{bc})\|_{0,{K_F}} \leq {C'_d}\,h_{K_F}^{-1}\|\psi_F \qvec_{bc}\|_{0,{K_F}},\label{eq:norm+_eta_f_ineq_7.1}\\
&{c_F}\|e_h\|_{0,F}^2\leq (e_h,\psi_F e_h)_{0,F},\label{eq:norm+_eta_f_ineq_5.1.F}\\
&\|\psi_F e_h\|_{0,F}\leq \|e_h\|_{0,F},\label{eq:norm+_eta_f_ineq_6.1.F}
\end{align}
with the constants ${c_F}$ and ${C'_d}$ depending only on the polynomial degree of $\tilde{\phi}_h$, $d$, and
$\kappa_{K_F}$. \\
Let ${\xi_{bc,F}}=(\psi_F\qvec_{bc},0)$ in ${K_F}$, and $0$ elsewhere. {By construction}, $\psi_F\qvec_{bc}$ is smooth in ${K_F}$ (a closed subset of $\mathbb{R}^d$).
Hence, $\xi_{bc,F}\in\Xcal_{K_F}$.
According again to Lemma~\ref{lemma:leandre_reconstruction}, {still using the notation ${\qvec_f} = \T_o\pvec_h +\H\grad\tilde{\phi}_h$}
\begin{align}
d(\zeta-\tilde{\zeta_h},\xi_{bc,F})&=-(\T_o\pvec_h + \H\grad\tilde{\phi}_h , \psi_F\qvec_{bc})_{0,{K_F}}{+(e_h, \psi_F(\qvec_{bc}\cdot\nvec))_{0,\pa\Omega}} \nonumber\\
&= -(\qvec_f, \psi_F\qvec_{bc})_{0,{K_F}}{+(e_h, \psi_F e_h)_{0,F}}.\label{eq:void_equality}
\end{align}
By definition \eqref{local-norm} of the $|\cdot|_{+,{K_F}}$ norm, it now follows that
\begin{align*}
&\qquad d(\zeta-\tilde{\zeta}_h,\xi_{bc,F})\\
&\leq |\zeta-\tilde{\zeta}_h|_{+,{K_F}}\|\xi_{bc,F}\|_S\nonumber\\
&\leq|\zeta-\tilde{\zeta}_h|_{+,{K_F}}
 \Big\{\|\delta_o^{1/2}(\psi_F \qvec_{bc})\|_{0,{K_F}}^2 + \delta^{max}_{o,{K_F}}h_{K_F}^2\|\dive (\psi_{F} \qvec_{bc})\|_{0,{K_F}}^2\nonumber \\
 & \hskip 27truemm +\sum_{F'\in{\mathcal{F}^e_h\cap\pa {K_F}}} \delta^{max}_{o,{K}_{F'}}h_{\perp F'}\|\tilde{\Gamma}_e^{1/2}(\psi_F \qvec_{bc}\cdot\nvec)\|_{0,F'}^2
  \Big\}^{1/2}\nonumber\\
&\leq|\zeta-\tilde{\zeta}_h|_{+,{K_F}}{ }
\Big\{  \delta^{max}_{o,{K_F}}(\|(\psi_F \qvec_{bc})\|_{0,{K_F}}^2 + { h_{K_F}^2} \| \dive (\psi_{F} \qvec_{bc})\|_{0,{K_F}}^2) \\
&\hskip 27truemm + \delta^{max}_{o,{K}_F}h_{\perp F} \|\tilde{\Gamma}_e^{1/2}(\psi_{F}{e_h})\|_{0,F}^2
 \Big\}^{1/2}\nonumber\\
&\leq|\zeta-\tilde{\zeta}_h|_{+,{K_F}}
 { \Big\{\delta^{max}_{o,{K_F}}\{1+{(C'_d)}^2\}}\|\qvec_{bc}\|_{0,{K_F}}^2 + \delta^{max}_{o,{K}_F}h_{\perp F}(\tilde{\Gamma}_e)^*\| e_h\|_{0,F}^2
  \Big\}^{1/2},
\end{align*}
where we used the inverse inequalities ~\eqref{eq:norm+_eta_f_ineq_6.1}, \eqref{eq:norm+_eta_f_ineq_7.1} and~\eqref{eq:norm+_eta_f_ineq_6.1.F} to reach the last line.

Let us now show that
\begin{equation}\label{eq:estimation_h_perpF}
\|\qvec_{bc}\|_{0,{K_F}}\leq  \frac1{\sqrt{3}}(h_{\perp F})^{1/2}\|e_h\|_{0,F}.
\end{equation}
{\small First, if ${K_F}$ is a rectangle or a cuboid, the definition of $\qvec_{bc}$ gives
\begin{align*}
 \|\qvec_{bc}\|_{0,{K_F}}^2&=\int_{K_F} e_h^2(\xvec+\lambda_F(\xvec)h_{\perp F}\nvec_F)|\boldsymbol{\rho}_{{K_F},F}(\xvec)|^2 d\xvec\\
 &=h_{\perp F}\int_{\lambda_F=0}^{1}\int_{\zvec\in F} e_h^2(\zvec)(1-\lambda_F)^2 d\lambda_Fd\zvec\\
 &=h_{\perp F}\|e_h\|_{0,F}^2\int_{0}^{1}t^2dt\\
 &=\frac{1}{3} h_{\perp F}\|e_h\|_{0,F}^2,
\end{align*}
where we used at the second line that $\zvec=\xvec+\lambda_F(\xvec)h_{\perp F}\nvec_F \in F$ for all $\xvec \in {K_F}$. \\
Second, if ${K_F}$ is a simplex, the definition of $\qvec_{bc}$ gives
\begin{align*}
&\qquad  \|\qvec_{bc}\|_{0,{K_F}}^2\\
 &=\int_{K_F} e_h^2(\sum_{i=1}^{d-1}\lambda_i(\xvec)\xvec_i +(\lambda_0(\xvec)+\lambda_d(\xvec))\xvec_d)|\boldsymbol{\rho}_{{K_F},F}(\xvec)|^2d\xvec\\
 &=d!\,|{K_F}|\int_{\lambda_0=0}^{1}\int_{\lambda_1=0}^{1-\lambda_0}\dots\int_{\lambda_d=0}^{1-\sum_{i=0}^{d-1}\lambda_i} e_h^2(\sum_{i=1}^{d-1}\lambda_i\xvec_i +(\lambda_0+\lambda_d)\xvec_d)(1-\lambda_0)^2 \prod_{i=0}^d d\lambda_i\\
 &\leq \frac{d!\,|{K_F}|}{(d-1)!\,|F|}\|e_h\|_{0,F}^2\int_{\lambda_0=0}^{1}(1-\lambda_0)^2 d\lambda_0\\
 &\leq \frac{d}{3} \frac{|{K_F}|}{|F|}\|e_h\|_{0,F}^2\\
 &\leq \frac{1}{3} h_{\perp F}\|e_h\|_{0,F}^2.
\end{align*}
Above,  to reach the third line, we used the fact that for $\lambda_0\in(0,1]$, $\{\sum_{i=1}^{d-1}\lambda_i\xvec_i $ $+ (\lambda_0+\lambda_d)\xvec_d\} \subsetneq F$.} \\
This completes the proof of~\eqref{eq:estimation_h_perpF}. We thus infer that
\begin{align}
&\qquad d(\zeta-\tilde{\zeta}_h,\xi_{bc,F}) \nonumber\\
&\leq|\zeta-\tilde{\zeta}_h|_{+,{K_F}}
  (\delta^{max}_{o,{K_F}_F}h_{\perp F})^{1/2}\Big\{\frac13 \{1+{(C'_d)}^2\}{+(\tilde{\Gamma}_e)^*}
  \Big\}^{1/2}\|e_h\|_{0,F},\label{eq:norm+_eta_f_ineq_44}
\end{align}
Using~\eqref{eq:void_equality},~\eqref{eq:norm+_eta_f_ineq_44} and {(\ref{estimate_on_qf})}, we obtain
\begin{align*}
&\quad (e_h, \psi_F e_h)_{0,F} \\
&\leq \|\qvec_{f}\|_{0,{K_F}}\|\qvec_{bc}\|_{0,{K_F}} \\
&\quad  +|\zeta-\tilde{\zeta}_h|_{+,{K_F}}
  (\delta^{max}_{o,{K}_F}h_{\perp F})^{1/2}\Big\{\frac13\{1+{(C'_d)}^2\} +(\tilde{\Gamma}_e)^*
  \Big\}^{1/2}\|e_h\|_{0,F}\\
& \leq |\zeta-\tilde{\zeta}_h|_{+,{K_F}}\Big[\frac{\{1+{C_d^2}\}^{1/2}}{c_q}(\delta^{max}_{o,{K_F}})^{1/2}\|\qvec_{bc}\|_{0,{K_F}}\\
&\hskip 27truemm +  (\delta^{max}_{o,{K_F}}h_{\perp F})^{1/2} \Big\{ \frac13\{1+{(C'_d)}^2\} +(\tilde{\Gamma}_e)^*
  \Big\}^{1/2}\|e_h\|_{0,F}\Big]\\
& \leq |\zeta-\tilde{\zeta}_h|_{+,{K_F}}\Big[\frac{\{1+{C_d^2}\}^{1/2}}{\sqrt{3}c_q}(\delta^{max}_{o,{K_F}}h_{\perp F})^{1/2} \\
&\hskip 27truemm + (\delta^{max}_{o,{K_F}}h_{\perp F})^{1/2}\Big\{ \frac13\{1+{(C'_d)}^2\} +(\tilde{\Gamma}_e)^*
  \Big\}^{1/2}\Big]\|e_h\|_{0,F}\\
 & \leq |\zeta-\tilde{\zeta}_h|_{+,{K_F}}
  (\delta^{max}_{o,{K_F}}h_{\perp F})^{1/2}\Big[\frac{\{1+{C_d^2}\}^{1/2}}{\sqrt{3}c_q}\\
  &\hskip 47truemm +\Big\{ \frac13\{1+{(C'_d)}^2\} +(\tilde{\Gamma}_e)^* \Big\}^{1/2}\Big]\|e_h\|_{0,F}.
\end{align*}
We concludes the proof by using~\eqref{eq:norm+_eta_f_ineq_5.1.F}.
\end{proof}
\begin{remark}
Assume in addition in Theorem~\ref{thm_apost_error_est} that there exists a constant $\kappa>0$, such that $\min_{K\in\Tcal_h} \kappa_K \geq \kappa$, for all $h>0$. Then, the constants $\mathtt{c}$ and $\mathtt{C}$ {only depend on $\kappa_K$}.
\end{remark}
}

The results of this section extend with the same arguments to the situation where $\T_e \geq 0$ may vanish if one slightly modifies the definition of the norms by
\begin{align*}
\|\zeta\|_{S,\star}^2 &= (\delta_o\pvec,\pvec)_{0,\Omega}
+(\delta_{\star}\,\phi,\phi)_{0,\Omega}\\
&\qquad 
+{\sum_{K\in \Tcal_h} \delta^{max}_{o,K}h_K^2 \| \dive \pvec\|_{0,K}^2} {+ \sum_{F\in\mathcal{F}^e_h} \delta^{max}_{o,K_F}h_{\perp F} \|\tilde{\Gamma}_e^{1/2}(\pvec\cdot\nvec)\|_{0,F}^2},\\
&|\zeta|_{+,\star,K}=\sup_{\xi\in \Xcal_K, \|\xi\|_{S,\star}\leq 1}d(\zeta,\xi),
\end{align*}
where $\delta_\star$ is defined by
\begin{equation*}
\delta_{\star}|_K=
\left\{
\begin{aligned}
& \delta_e \quad \text{if }\inf_K \|\delta_e\|>0,\\
&\mathbb{I}  \quad \text{otherwise.}
\end{aligned}
\right.
\end{equation*}

Let us define for all $K\in \Tcal_h,$
\[ \delta^{max}_{\star,K} = \max_{g\in\Ical_G, i\in\Ical_e}\sup_{K}(((\delta_{\star})_{g,g})_{i,i}), \quad \T^{min}_{\star,K} = \min_{g\in\Ical_G, i\in\Ical_e}\inf_{K}(((\delta_{\star})_{g,g})_{i,i}).
\]

Under the assumptions of Theorem~\ref{theorem:norm+K}, one has the reliability estimate
\begin{align*}
&|\zeta-\tilde{\zeta}_h|_{+,K}\leq \left({\eta}^2_{r,\star,K} +\sum_{K'\in N(K)}\eta^2_{f,K'}{+\sum_{F\in {\mathcal{F}_h^e}\cap \pa K}\eta_{bc,F}^2}\right)^{1/2},
\end{align*}
where the residual estimator becomes  
\begin{equation*}
{\eta_{r,\star,K}}=\|\delta_{\star}^{-1/2} (S_f-\H^T\dive \pvec_h -\T_e\tilde{\phi}_h)\|_{0,K}.
\end{equation*}
Under the assumptions of Theorem~\ref{thm_apost_error_est}, one has the efficiency estimates {for $K\in\Tcal_h$ and $F\in \mathcal{F}^e_h$},
\begin{align*}
{\eta}_{r,\star, K}&\leq \mathtt{c}\,{ \left(\frac{\delta^{max}_{\star,K}}{\delta^{min}_{\star,K}}\right)^{1/2}}\,|\zeta-\tilde{\zeta}_h|_{+,\star,K}
, \\
{{\eta}_{f,K}}&{\leq { \mathtt{C} \left(\frac{\delta^{max}_{o,K}}{\delta^{min}_{o,K}}\right)^{1/2}}\,|\zeta-\tilde{\zeta}_h|_{+,\star,K}},\\
{{\eta}_{bc,F}}&{\leq { \mathsf{C} }\,|\zeta-\tilde{\zeta}_h|_{+,\star,K_F}.}
\end{align*}

\subsection{Extension to the Domain Decomposition+$L^2$-jumps method}
According to~\cite{CiDoGeMa25}, {it is possible to} extend in this section the strategy to a domain decomposition method introduced in~\cite{CiJK17}, namely the DD+$L^2$-jumps method. We recall here the definition of this {multi-domain} approach presented in~\cite[Section 2]{CiDoGeMa25}.

To this aim, let us consider a partition $\{\Omega^*_{i^*}\}_{1\leq {i^*}\leq {N^*}}$ of $\Omega$ which is independent of the physical partition $\{{\Omega}_i\}_{1\leq i\leq N}$ introduced in Section~\ref{sec-notations}. For a field $v$ defined over $\Omega$, we shall use the notation $v_{i^*}=v|_{\Omega^*_{i^*}}$, for $1\leq {i^*}\leq {N^*}$. We denote by $\Gamma_{{i^*}{j^*}}$ the interface between two subdomains $\Omega^*_{i^*}$ and $\Omega^*_{j^*}$ for ${i^*}\neq {j^*}$: if $ \text{dim}_{H}\left(\partial\Omega^*_{i^*}\cap\partial\Omega^*_{j^*}\right)=d-1$, then $\Gamma_{{i^*}{j^*}}=\text{int}(\partial\Omega^*_{i^*}\cap\partial\Omega^*_{j^*})$; otherwise, $\Gamma_{{i^*}{j^*}}=\emptyset$. By construction, $\Gamma_{{i^*}{j^*}}=\Gamma_{{j^*}{i^*}}$.
 We define the {global} interface $\Gamma$ by
\[
\Gamma=\cup_{{i^*}=1}^{N^*}\cup_{{j^*}={i^*}+1}^{N^*}\overline{\Gamma_{{i^*}{j^*}}}.
\]
{For $1\leq {i^*}\leq {N^*}$, we let $\Gamma_{i^*} = \pa\Omega_{i^*}\cap \pa\Omega$}.
We then introduce the function spaces
\begin{align*}
\PQvec(\Omega) &= \{\qvec\in L^2(\Omega)\ | \ \qvec_{i^*}\in {\Hvec}(\dive,\Omega_{i^*}),\ {(\qvec\cdot\nvec)_{|\Gamma_{i^*}} \in L^2(\Gamma_{i^*})},\quad 1\leq {i^*}\leq {N^*}\},\\
M&=\{m =(m_{{i^*}{j^*}})_{{i^*}<{j^*}} \in \prod_{{i^*}<{j^*}}L^2(\Gamma_{{i^*}{j^*}})\},\\
\mathbf{Q}^*&=\{\qvec\in \PQvec(\Omega)\ | \ [\qvec\cdot \nvec ]\in M\},\\
\Wens&=\mathbf{Q}^*\times L^2(\Omega)\times M,
\end{align*}
where $[\qvec\cdot \nvec ]$ stands for {the {\em global jump} of the normal component} and is defined by
\begin{align*}
[\qvec\cdot \nvec ]|_{\Gamma_{{i^*}{j^*}}} = {\qvec}_{i^*}\cdot \nvec_{i^*} + \qvec_{j^*}\cdot\nvec_{j^*}, \text{ for } 1\leq {i^*} < {j^*} \leq {N^*}.
\end{align*}
These spaces are endowed with their natural norm, eg.
\[ \|m\|_M = \left(\sum_{1\le{i^*}<{j^*}\le {N^*}}\|m_{{i^*}{j^*}}\|_{0,\Gamma_{{i^*}{j^*}}}^2\right)^{1/2}. \]
The variational formulation associated to the multi-domain problem writes
\begin{align}
\left\{
\begin{aligned}
&\text{Find } \uelt=(\pvec,\phi,\ell)\in \Wens \text{ such that for all } \welt=(\mathbf{q},\psi,m)\in \Wens,\\
& \quad c_{DD}(\uelt,\welt) = f({\welt}).
\end{aligned}
\right.
\label{eq:varf_ddm}
\end{align}
Extending the definition \eqref{eq:bilin-form-c-PC} of the bilinear form $c$ to piecewise smooth fields by replacing $\displaystyle\int_{\Omega}$ by $\displaystyle \sum_{i^*=1}^{N^*}\int_{\Omega_{i^*}}$, one uses the forms
\[ c_{DD}(\uelt,\welt) = c((\pvec,\phi),(\qvec,\psi))
+ \int_{\Gamma} [\pvec\cdot \nvec] m - \int_{\Gamma} [\mathbf{q}\cdot \nvec] \ell,
\quad {\mbox{and } f(\welt) = (S_f, \psi)_{0,\Omega}}. \]
In addition to the physical variables $\pvec$ and $\phi$, the field $\ell$ can be seen as a Lagrange multiplier. {With the help of the appendix of \cite{CiJK17}, one is able to prove} there is equivalence between the multi-domain problem associated to~\eqref{eq:varf_ddm} and the mono-domain Problem~\eqref{eq:diff-mixed}.
For $K\in\Tcal_h$, we now introduce ${N^*}(K)=N(K)\cap \overline{\Omega^*_{K}}$ where $\Omega^*_{K}$ is the subdomain which includes $K$ and \[\Xcalmg_K^*  = \left\{\zeta=(\pvec,\phi)\in \PQvecmg(\Omega)\times \Ludud^2(\Omega)\ |\ \text{Supp}(\phi) \subset K, \text{Supp} (\pvec) \subset {N^*}(K)\right\}. \] 
Indeed, since $\pvec$ is in $\PQvecmg(\Omega)$, only the mesh elements $K'$ of $N(K)$ that belong to $\overline{\Omega^*_K}$ have to be considered above. In this sense, the definition is slightly different from the one given in the mono-domain case: $N(K)$ is now replaced by ${N^*}(K)$, {because there is no continuity of the normal trace across $\Gamma$}. Then one can define the following $\Xcalmg_K^*$-local norm, for all $\zeta\in \Xcalmg$,
\begin{equation} \label{ddm_local-norm}
|\zeta|_{+,K}=\sup_{\xi\in \Xcalmg_K^*, \|\xi\|_S\leq 1}d(\zeta,\xi).
\end{equation}

 We introduce discrete, finite-dimensional, spaces indexed by $h$ as follows:
${\Qvec}_{{i^*},h}\subset\Hvec(\dive,\Omega^*_{i^*})$ 
and $L_{{i^*},h}\subset L^2(\Omega^*_{i^*})$, for $1\leq {i^*} \leq {N^*}$. {In the spirit of the mono-domain case}, we impose the following requirements for all $1\leq {i^*} \leq {N^*}$:
\begin{itemize}
	\item $\qvec_{{i^*},h}\cdot \nvec \in L^2(\partial\Omega^*_{i^*})$ for all $h>0$, for all $\qvec_{{i^*},h}\in {\Qvec}_{{i^*},h}$;
	\item $\dive  {\Qvec}_{{i^*},h} \subset L_{{i^*},h}$ for all $h>0$;
	\item $({\Qvec}_{{i^*},h})_h$ and $(L_{{i^*},h})_h$ satisfy the approximability property~\eqref{eq:approx-pro} in $\Omega^*_{i^*}$.
\end{itemize}
We observe that, to build conforming discretizations in $\PQvec(\Omega)$, one uses meshes that are {\em conforming} with respect to {every subdomain $\Omega^*_{i^*}$} of the partition. {Hence}, one first defines, for $1 \le {i^*} \le {N^*}$, families of {\em conforming} meshes $(\Tcal_{h,{i^*}})_h$ of $\overline{\Omega^*_{i^*}}$. Then, the meshes $(\Tcal_h)_h$ are built by aggregating for given $h$ the meshes $(\Tcal_{h,{i^*}})_{1 \le {i^*} \le {N^*}}$. \\
If $\Omega^*_{i^*}$ and $\Omega^*_{j^*}$ share a common (non-empty) interface $\Gamma_{{i^*}{j^*}}$, the meshes $\Tcal_{h,{i^*}}$ and $\Tcal_{h,{j^*}}$ are said to be {\em matching} if their restriction to $\Gamma_{{i^*}{j^*}}$ coincide. Otherwise, they are {\em non-matching}. As soon as there is a pair of non-matching meshes, the mesh $\Tcal_h$ is not {\em conforming}: we call this situation the {\em non-matching case}. On the contrary, when all pairs of meshes are matching,  $\Tcal_h$ itself is a {\em conforming} mesh {with respect to $\Omega$}: we call this situation the {\em matching case}. \\
Introducing the discrete space of Lagrange multipliers $M_h\subset M$, we then set 
\begin{align*}
\Qvec^*_h=\prod_{{i^*}=1}^{N^*}{\Qvec}_{{i^*},h},\quad {L^*_h}=\prod_{{i^*}=1}^{N^*}L_{{i^*},h},\quad
\Wens_h=\Qvec^*_h\times  L^*_h\times M_h,
\end{align*}
{For $1\le{i^*}\le{N^*}$, we introduce the spaces of (discrete) normal traces}
\begin{align*}
T_{{i^*},h}=\{t_{{i^*},h}\in L^2(\partial\Omega^*_{i^*}\cap\Gamma)\ |\ \exists \qvec_{{i^*},h}\in\Qvec_{{i^*},h},\ t_{{i^*},h}=\qvec_{{i^*},h}\cdot \nvec_{{i^*}_{|\partial\Omega^*_{i^*}\cap\Gamma}}\}.
\end{align*}
We {further} assume that the space of piecewise constant fields is included in $M_h$. \\
The discrete variational formulation associated to~\eqref{eq:varf_ddm} writes
\begin{align}
\left\{
\begin{aligned}
&\text{Find }\uelt_h=(\pvec_h,\phi_h,\ell_h)\in \Wens_h \text{ such that for all }\welt_h =(\qvec_h,\psi_h,m_h)
\in \Wens_h,\\
& c_{DD}(\uelt_h,\welt_h) =  f({\welt_h}).
\end{aligned}
\right.
\label{eq:ddm_varf_discrete}
\end{align}
Following~\cite[Section 5]{CiJK17}, we define the discrete $L^2$-projection operators $ (\Pi_{{i^*}})_{1\le{i^*}\le{N^*}}$ from the spaces of normal traces {$(T_{i^*,h})_{1\le{i^*}\le{N^*}}$} to $M_h$,\footnote{More precisely, from $T_{{i^*},h}$ to $\{m_h\in M_h\ |\ \mbox{supp}(m_h)\subset\partial\Omega^*_{i^*}\cap\Gamma\}$.} resp. the discrete $L^2$-projection operators $(\pi_{{i^*}})_{1\le{i^*}\le{N^*}}$ from $M_h$ to $(T_{{i^*},h})_{1\le{i^*}\le{N^*}}$. {For $1\le{i^*}\le{N^*}$, they} are defined by
\begin{align*}
&\forall t_{{i^*},h}\in T_{{i^*},h},\ \forall m_h\in M_h, \qquad\left\{
\begin{aligned}
\int_{\partial\Omega^*_{i^*}\cap\Gamma}(\Pi_{i^*}t_{{i^*},h}-t_{{i^*},h})m_h&=0\\
\int_{\partial\Omega^*_{i^*}\cap\Gamma}(\pi_{i^*} m_h -m_h)t_{{i^*},h}&=0.
\end{aligned}
\right.
\end{align*}
Next, let $\pvec_h\in\Qvec^*_h$. {For ${i^*}<{j^*}$}, we define the {\em discrete jump} of the normal component of $\pvec_h$ on the interface $\Gamma_{{i^*}{j^*}}$ as $[\pvec_h\cdot \nvec]_{h,{i^*}{j^*}} := \Pi_{{i^*}} (\pvec_{{i^*},h}\cdot \nvec_{i^*}{}_{|\Gamma_{{i^*}{j^*}}}) + \Pi_{{j^*}} (\pvec_{{j^*},h}\cdot \nvec_{j^*}{}_{|\Gamma_{{i^*}{j^*}}})$. {Then, the {\em discrete global jump} $[\pvec_h\cdot \nvec ]_h$ is defined by
\begin{align*}
[\pvec_h\cdot \nvec ]_h|_{\Gamma_{{i^*}{j^*}}} = [\pvec_h\cdot \nvec]_{h,{i^*}{j^*}}, \text{ for } 1\leq {i^*} < {j^*} \leq {N^*}.
\end{align*}}

\begin{assumption}\label{assumption:discrete_ddm}
	We assume that there exists $\beta_h>0$ such that for all $\qvec_h\in \Qvec^*_h$, 
	\begin{align}
	\int_{\Gamma} [\qvec_h\cdot \nvec]_h[\qvec_h\cdot \nvec]\geq \beta_h  \int_{\Gamma} [\qvec_h\cdot \nvec]^2,
	\label{hyp:beta_h}
	\end{align}
	and that there exists $\gamma_h>0$ such that for all $m_h\in M_h$,
	\begin{align}
	\sum_{{i^*}=1}^{N^*}\sum_{{j^*}={i^*}+1}^{N^*}\int_{\Gamma_{{i^*}{j^*}}}((\pi_{i^*}m_h)^2+(\pi_{j^*}m_h)^2) \geq \gamma_h \|m_h\|_M^2.
	\label{hyp:gamma_h}
	\end{align}
\end{assumption}
We refer to \cite[Section 5.2]{CiJK17} for an extensive discussion on how to fulfill this assumption in practice. In particular (see \S5.2.1 in \cite{CiJK17}), the choice 
\begin{equation}\label{Mh_SufficientCondition}
 M_h=\sum_{{i^*}=1}^{N^*}T_{{i^*},h}
\end{equation}
{can be shown to be a sufficient condition for Assumption \ref{assumption:discrete_ddm} to hold. Then, adapting the proof given in~\cite[Section 5.1]{CiJK17} to cover the case of a Robin boundary condition, one finds that}, under {Assumption~\ref{assumption:discrete_ddm}}:
\begin{itemize}
\item the discrete problem \eqref{eq:ddm_varf_discrete} is well-posed\,;
\item the discrete solution fulfills $[\pvec_h\cdot \nvec] = 0$, so that $\pvec_h\in \Qududvec(\Omega)$.
\end{itemize}

Before stating the a posteriori estimates, we define a reconstruction associated to the DD$+L^2$ jumps method. We choose {the method proposed in \cite{CiDoGeMa25}. Precisely, we} look for $\tilde{\zeta}_h:=\tilde{\zeta}_h(\pvec_h,\phi_h,\ell_h)\in \Qududvec(\Omega)\times {\Vudud}$ where $(\pvec_h,\phi_h,\ell_h)$ is the discrete solution to~\eqref{eq:ddm_varf_discrete}. {In particular, there are only two components appearing in the resconstruction $\tilde{\zeta}_h$}.
Since under Assumption~\ref{assumption:discrete_ddm}, one has $\pvec_h\in\Qududvec(\Omega)$, one can set $\tilde{\pvec}_h= \pvec_h$. Finally, we will design $\tilde{\phi}_h$ as a function of $(\phi_h,\ell_h)$. To summarize, we will consider from this point on reconstructions like
\[ \tilde{\zeta}_h=(\pvec_h,\tilde{\phi}_h(\phi_h,\ell_h)) \in \Qududvec(\Omega)\times {\Vudud}. \]
We refer to~\cite[Section 6.1]{CiDoGeMa25} for the definition of reconstruction approaches, and their practical implementation.

\begin{theorem}\label{theorem:ddm_norm+K}
We suppose that Assumption~\ref{assumption:discrete_ddm} holds. 
Let $\tilde{\zeta}_h=(\pvec_h,\tilde{\phi}_h)\in \Qududvec(\Omega)\times {\Vudud}$ be a reconstruction.
For any $K\in \Tcal_h$, we define the residual estimator ${\eta_{r,K}}$ as in \eqref{eq:def_residual_estimator}, the flux estimator ${\eta_{f,K}}$ as in \eqref{eq:def_flux_estimator}. 
For any $F\in{\mathcal{F}_h^e}$, we define the Robin boundary condition estimator ${\eta_{bc,F}}$ as in~\eqref{eq:def_bc_estimator}. One has the reliability estimate
\begin{align}
&|\zeta-\tilde{\zeta}_h|_{+,K}\leq \left({\eta}^2_{r,K} +\sum_{K'\in {N^*}(K)} \eta^2_{f,K'}{+\sum_{F\in {\mathcal{F}_h^e}\cap \pa K}\eta_{bc,F}^2}\right)^{1/2}.\label{eq:ddm_rt_ineq_1}
\end{align}
\end{theorem}
\begin{proof}
The proof is similar to the proof of~\cite[Theorem 6.4]{CiDoGeMa25}.
\end{proof}
\begin{theorem}[local efficiency of the {\em a posteriori} error estimators]\label{ddm_thm_apost_error_est}
Let {Assumptions~\ref{assumption:locality_polynomial} and~\ref{assumption:discrete_ddm} hold}. 
Let $\tilde{\zeta}_h=(\pvec_h,\tilde{\phi}_h)\in \Qududvec(\Omega)\times {\Vudud}$ be a reconstruction.
For $K \in \Tcal_h$, let ${\eta}_{r,K}$ and ${\eta}_{f,K}$ be the residual and flux estimators respectively given by~\eqref{eq:def_residual_estimator}, and~\eqref{eq:def_flux_estimator}. Estimates~\eqref{eq:norm+_local_efficiency_eta_r} and~\eqref{eq:norm+_local_efficiency_eta_f} hold true 
where {$\mathtt{c}$ and $\mathtt{C}$ are constants which depend} only on the polynomial degree of $S_f$, $\T_o$, $\T_e$  {and $\tilde{\phi}_h$}, $d$, and the shape-regularity parameter $\kappa_K$. 
\\
{For $F \in \mathcal{F}^e_h$, let ${\eta}_{bc,F}$ be the Robin boundary condition estimator given by~\eqref{eq:def_bc_estimator}. 
Estimates \eqref{eq:norm+_local_efficiency_eta_bc} holds true
where 
{$\mathtt{c}$ and $\mathtt{C}$ are constants which depend} only on the polynomial degree of $S_f$,  $\T_e$  {and $\tilde{\phi}_h$}, $d$, and the shape-regularity parameter $\kappa_{K_F}$.
}
\end{theorem}
\begin{proof}
The proof is completely similar to the proof of Theorem~\ref{thm_apost_error_est}.
\end{proof}
\everymath{\displaystyle} 
\section{Numerical experiment}
\label{sec:numerical_results}

In this section, we illustrate numerically the use of the a posteriori estimators
devised in the previous section. To this aim, we present an example of Adaptive
Mesh Refinement (AMR) on a source problem inspired by the Model 1 case 2 test case defined in~\cite{Ta91}.

\tikzset{every picture/.style={line width=0.75pt}} 

Section~\ref{sec:num_AMR} defines the adaptive mesh refinement.
Section~\ref{subsec:test_case} describes the setting of the test case. Section~\ref{subsec:results} shows the numerical results obtained.

\subsection{Adaptive mesh refinement}
\label{sec:num_AMR}
In this Section, we recall e.g. from~\cite[Section 6]{CDM25} a classical definition of an AMR strategy. 
This iterative process is divided into four modules as presented in Figure~\ref{fig:AMR_algo}, where $\eps_{\text{AMR}}>0$ is a user-defined parameter, {that accounts for the maximal element-wise tolerance error. Precisely, 
we use a relative stopping criterion which writes $\eps_{\text{AMR}}=\eps_{\text{AMR, rel}}\|\phi_h\|_{L^2(\Omega)}$, where $\eps_{\text{AMR, rel}}>0$. Each module is described below in the mono-domain setting. The extension to the multi-domain approach is then outlined. We recall that for simplicity, we present the algorithm in the case where the vacuum boundary condition is prescribed everywhere on $\partial\Omega$. The methodolody easily extends to the case where mixed boundary conditions on $\pa\Omega$ described in Appendix~\ref{sec:appendix_mixed}.
}
\begin{figure}[htbp]
\resizebox{\textwidth}{!}{
\begin{tikzpicture}
\node [draw] (Im) at (0,2) {{Initial mesh}};
\node (S) at (0,1) {\textbf{SOLVE}};
\node (E) at (3,1) {\textbf{ESTIMATE}};
\node (T) at (7,1) {$\ds \max_{K \in\Tcal_{h}} \eta_K \leq \eps_{\text{AMR}}$?};
\node (M) at (10.5,1) {\textbf{MARK}};
\node (Mv) at (9.25,0.75) {No};
\node[rotate=21] (Mv) at (8.75,1.85) {Yes};
\node (R)  at (13,1) {\textbf{REFINE}};
\node (Rf)  at (13.2,0.8) { };
\node[draw] (Rd)  at (10,2) {Stop};
\draw[->,>=latex] (Im) -- (S);
\draw[->,>=latex] (S) -- (E);
\draw[->,>=latex] (E) -- (T);
\draw[->,>=latex] (T) -- (M);
\draw[->,>=latex] (M) -- (R);
\draw[->,>=latex] (T) -- (Rd) ;
\draw[->,>=latex] (Rf) to[bend left=12] (S);
\end{tikzpicture}
}
\caption{Description of the AMR process.}
    \label{fig:AMR_algo}
\end{figure}
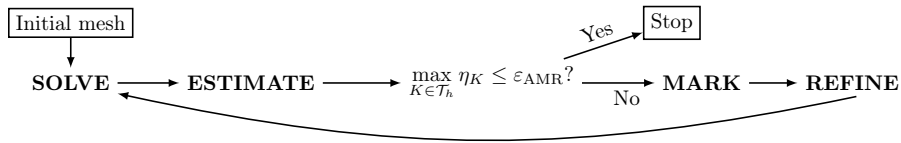

\subsubsection{SOLVE module}
For the source problem, the  \textbf{SOLVE} module amounts to solving the discrete problem~\eqref{eq:VF-1h-PC}. 
\subsubsection{ESTIMATE module}
In the  \textbf{ESTIMATE} module, the  local error indicator $\eta_K$ is computed on each mesh element $K$. Using the {\em a posteriori} error estimate~\eqref{eq:rt_ineq_1}, this error indicator is defined by
\begin{align}\label{LocalIndicator_mono}
\eta_K:= \left({\eta}^2_{r,K} +\sum_{K'\in N(K)}\eta^2_{f,K'}{+\sum_{F\in {\mathcal{F}^e_h}\cap \pa K}\eta_{bc,F}^2}\right)^{1/2}.
\end{align}

\subsubsection{MARK module}
The purpose of the  \textbf{MARK} module is to select a set of mesh elements with large error: then, these elements are refined.
In other words, the marking strategy consists in selecting a set of elements {$S\subset\Tcal_{h}$} of minimal cardinal such that one has $$\eta(S) \simeq \theta \, \eta(\Tcal_{h}), \quad \text{where} \  \eta(S) := \left(\sum_{K \in S} \eta_K^2\right)^{1/2}, \quad \mbox{resp.} \ \eta(\Tcal_{h}) := \left(\sum_{K \in \Tcal_{h}} \eta_K^2\right)^{1/2}$$ and $\theta>0$ is a user-defined parameter.
According to~\cite[Section 6]{CDM23}, an efficient 
strategy which preserves the Cartesian structure of the mesh is the \emph{direction} marker strategy. One selects for each direction $\mathbf{e}_i$, $i=1,\dots,d$, the smallest set of lines $L_{i}$ along that direction such that $\eta(L_{i}) \geq \theta \eta(\Tcal_{h})$. The resulting selected set is $\cup_{i=1,\dots,d} L_{i}$.

\subsubsection{REFINE module}\label{ss-sec_REFINEModule}
The \textbf{REFINE} module refines the mesh $\Tcal_{h}$ if the stopping criterion $\ds\max_{K \in\Tcal_{h}} \eta_K \leq \eps_{\text{AMR}}$ is not reached. 

\subsubsection{Extension to the Domain Decomposition+$L^2$-jumps method}
\label{sec:num_AMR_ddm}
The modules slightly differs in the case of the Domain Decomposition+$L^2$-jumps method.  
The discrete {multi-domain} problem~\eqref{eq:ddm_varf_discrete} is solved in the  \textbf{SOLVE} module.
Using the {\em a posteriori} error estimate~\eqref{eq:ddm_rt_ineq_1}, the local error indicator is now defined by for each $K\in\Tcal_h$ by
\begin{align}\label{LocalIndicator_multi}
\eta_K:= \left({\eta}^2_{r,K} +\sum_{K'\in N^*(K)}\eta^2_{f,K'}{+\sum_{F\in {\mathcal{F}^e_h}\cap \pa K}\eta_{bc,F}^2}\right)^{1/2}.
\end{align}
{The main difference with the mono-domain setting is that} the \textbf{MARK} module is applied \textit{independently} on each subdomain $\Omega^*_{i^*}$ with a user-defined parameter $\theta_{i^*}$, for all $1\leq i^*\leq N^*$.
{In addition}, the module \textbf{REFINE} refines, for all $1\leq i^*\leq N^*$,  the mesh $\Tcal_{h, i^*}$ if the stopping criterion is not reached locally i.e. $\ds  \max_{K \in\Tcal_{h, i^*}} \eta_K > \eps_{\text{AMR}}$.
 
\subsection{Setting of the test case}\label{subsec:test_case}
{Lengths are given in centimeters. We solve the SP$_1$ problem~\eqref{eq:diff_primal} 
in the domain $\Omega=(0,25)^3$, which is made of three different materials (core, control rod, reflector). The core is located in the region $(0,15)^3$, the control rod is located in the region $(15,20)\times(0,5)\times(0,25)$, while the reflector is in the rest of $\Omega$. The geometry is depicted in Figure~\ref{fig:TakedaM1-geom} (side and top views)}. At the boundary, vacuum and reflection (homogeneous Neumann) boundary conditions are imposed. It corresponds to the case where mixed boundary conditions are imposed on the boundary detailed in Appendix~\ref{sec:appendix_mixed}.
Rather than considering the eigenvalue problem, we consider a source problem, where the source is defined in Table~\ref{table:TakedaCase1-phy-params}.
\begin{figure}[ht]
    \centering
    \resizebox{0.9\textwidth}{!}{
    \hspace{-1mm}
    \resizebox{0.07\textwidth}{!}
    {
    \begin{subfigure}[b]{0.4\textwidth}
        \raggedleft
        \begin{tikzpicture}[scale=4./25]
        \draw[->]   (0,0) -- (0,28);
        \draw[->]   (0,0) -- (28,0);
        \draw   (0,0) -- (0,25) -- (25,25) -- (25,0) -- cycle ;
       \draw  [fill={rgb, 255:red, 85; green, 160; blue, 255 }  ,fill opacity=1 ] (0,0) -- (0,25) -- (25,25) -- (25,0) -- cycle ;
        \draw  [fill={rgb, 255:red, 179; green, 173; blue, 0 }  ,fill opacity=1 ] (0,0) -- (0,15) -- (15,15) -- (15,0) -- cycle ;
        \draw  [fill={rgb, 255:red, 155; green, 155; blue, 155 }  ,fill opacity=1 ] (15,0) -- (15,25) -- (20,25) -- (20,0) -- cycle ;
        \draw (0,15) node [anchor= east]   {$15$};
        \draw (0,25) node [anchor= east]   {$25$};
        \draw (0,0) node [anchor= north east]   {$0$};
        \draw (15,0) node [anchor= north]   {$15$};
        \draw (20,0) node [anchor= north]   {$20$};
        \draw (25,0) node [anchor= north]   {$25$};
        \draw (28,0) node [anchor= west]   {$x$};
        \draw (0,28) node [anchor= south]   {$z$};
        \draw (-5,17.5) node [anchor= east, rotate=90]   {Reflection};
        \draw (12.5,-2.5) node [anchor= north]   {Reflection};
        \draw (27.5,17.5) node [anchor= east, rotate=90]   {{Vacuum}};
        \draw (12.5,27.5) node   {{Vacuum}};
        
        \end{tikzpicture}
        \caption{Radial view}
    \end{subfigure}
    }\hspace{1mm}
    \resizebox{0.07\textwidth}{!}
    {
    \begin{subfigure}[b]{0.4\textwidth}
        \raggedleft
        \begin{tikzpicture}[scale=4./25]
        \draw[->]   (0,0) -- (0,28);
        \draw[->]   (0,0) -- (28,0);
        \draw   (0,0) -- (0,25) -- (25,25) -- (25,0) -- cycle ;
       \draw  [fill={rgb, 255:red, 85; green, 160; blue, 255 }  ,fill opacity=1 ] (0,0) -- (0,25) -- (25,25) -- (25,0) -- cycle ;
        \draw  [fill={rgb, 255:red, 179; green, 173; blue, 0 }  ,fill opacity=1 ] (0,0) -- (0,15) -- (15,15) -- (15,0) -- cycle ;
        \draw  [fill={rgb, 255:red, 155; green, 155; blue, 155 }  ,fill opacity=1 ] (15,0) -- (15,5) -- (20,5) -- (20,0) -- cycle ;
        \draw (0,5) node [anchor= east]   {$5$};
        \draw (0,15) node [anchor= east]   {$15$};
        \draw (0,25) node [anchor= east]   {$25$};
        \draw (0,0) node [anchor= north east]   {$0$};
        \draw (15,0) node [anchor= north]   {$15$};
        \draw (20,0) node [anchor= north]   {$20$};
        \draw (25,0) node [anchor= north]   {$25$};
        \draw (28,0) node [anchor= west]   {$x$};
        \draw (0,28) node [anchor= south]   {$y$};
        \draw (12.5,-2.5) node [anchor= north]   {Reflection};
        \draw (-5,17.5) node [anchor= east, rotate=90]   {Reflection};
        \draw (12.5,-2.5) node [anchor= north]   {Reflection};
        \draw (27.5,17.5) node [anchor= east, rotate=90]   {{Vacuum}};
        \draw (12.5,27.5) node   {{Vacuum}};
        \end{tikzpicture}
        \caption{Axial view}
    \end{subfigure}
    }\hspace{1mm}
    \resizebox{0.02\textwidth}{!}
    {
    \begin{subfigure}[b]{0.13\textwidth}
    \raggedleft
    \begin{tikzpicture}[scale=0.4]
    \definecolor{couleur1}{RGB}{85, 160, 255}    
    \definecolor{couleur2}{RGB}{179, 173, 0}    
    \definecolor{couleur3}{RGB}{155, 155, 155}    

    \fill[white] (0, 1) rectangle (1, 10.);
        \node[right] at (1.2, -1 +10.25 ) {Reflector};
        \node[right] at (1.2, -2 +10.25 ) {Core};
        \node[right] at (1.2, -3 +10.25 ) {Control Rod};
    \foreach \i/\couleur in {1/couleur1, 2/couleur2, 3/couleur3} {
        \fill[\couleur] (0, -\i +10 ) rectangle (1, -\i +10.5);
    }
\end{tikzpicture}
\caption{Legend}
\end{subfigure}
}
}
    \caption{The benchmark {geometry}.}
    \label{fig:TakedaM1-geom}
\end{figure}
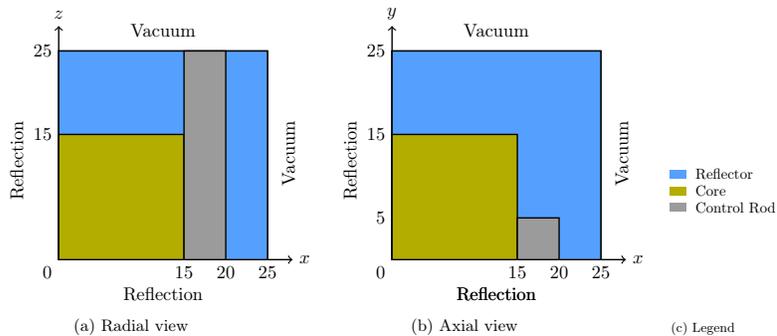

\begin{table}[!ht]
\centering
\begin{equation*}
\begin{array}{c|ccc}
\text{Source} & \text{Reflector} & \text{Core} & \text{Control rod} \\
  \toprule[\heavyrulewidth]\toprule[\heavyrulewidth]
S_f^1
& 0. & 9.09319 \times 10^{-3} & 0. \\[0.8em]
S_f^2
& 0. & 2.90183 \times 10^{-1} & 0. \\[0.8em]
\end{array}
\end{equation*}
\caption{Values of the source for the test case.}
\label{table:TakedaCase1-phy-params}
\end{table}

The reference solution is computed on a uniform mesh consisting of $80 \times 80 \times 80$ cells. The mesh step of {this reference grid is equal to $0.3125$}.

We compare three different refinement strategies: uniform refinement, AMR
with a mono-domain discretization~\cite{CDM23} and AMR with {the multi-domain approach} (DD$+L^2$ jumps method~\cite{CiDoGeMa25}). The initial mesh is {uniform} and consists of $5\times 5 \times 5$ cells. The initial mesh size is equal to $h = 5$ and the discretization is performed with RTN$_0$ and $\mathbb{Q}_0$ finite elements. In the DD$+L^2$ jumps method, we set $M_h$ as in (\ref{Mh_SufficientCondition}). The stopping criterion is set to $ \eps_\text{AMR,rel}=4.10^{-3} $.

\subsubsection{The mono-domain {setting}}
The AMR process for the mono-domain formulation {is applied} as described in Section~\ref{sec:num_AMR}. In the \textbf{ESTIMATE} module, the reconstruction is computed with the averaging method described in~\cite[Section 5.1.1]{CDM23}. {We study two configurations, denoted MONO-1, MONO-2: the value of the refinement parameter $\theta$ in the \textbf{REFINE} module is given in Table~\ref{tab:theta_monodomain}}.
\begin{table}[!ht]
\centering
\begin{tabular}{|c|c|}
\hline
 Configuration & $\theta$ \\
\hline
 MONO-1 & $0.5$ \\
 MONO-2 & $0.2$ \\
 \hline
\end{tabular}
\caption{AMR parameter {(mono-domain setting)}.}
\label{tab:theta_monodomain}
\end{table}

\subsubsection{The {multi-domain approach}}
Correspondingly, {we study two multi-domain configurations to perform} the AMR process defined in Section~\ref{sec:num_AMR_ddm} for the DD$+L^2$ jumps method. The subdivision into subdomains is designed so that each interface between two materials is also an interface for the domain decomposition:
\[
\begin{aligned}
&\Omega_1 = (0,15) \times  (0,15) \times (0,15), \\
&\Omega_2 = (15,20) \times (0,5) \times  (0,25), \\
&\Omega_3 = (20,25) \times (0,25) \times (0,25), \\
&\Omega_4 = (15,20) \times (5,25) \times (0,25), \\
&\Omega_5 = (0,15) \times  (15,25) \times (0,25), \\
&\Omega_6 = (0,15) \times  (0,15) \times (15,25).
\end{aligned}
\]
The {subdomains and the initial mesh are} represented in Figure~\ref{fig:TakedaM1C2-DDM-initial}.
\begin{figure}[htbp]
    \centering
    \resizebox{0.9\textwidth}{!}{
    \hspace{-3mm}
    \resizebox{0.03\textwidth}{!}{
    \begin{subfigure}{0.4\textwidth}
\begin{tikzpicture}
    \node at (0,0) {\includegraphics[width=5.cm]{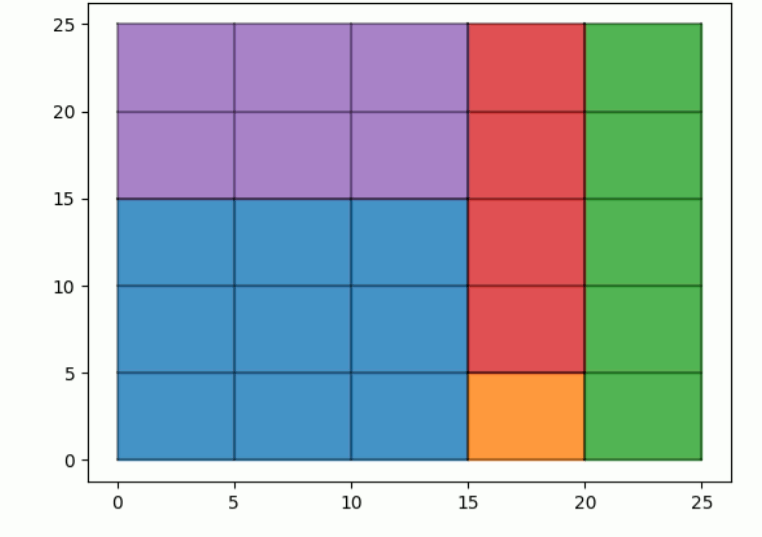}};
    \node[left] at (-2.25,0.125) {$y$};
    \node[below] at (0.125,-1.75) {$x$};
\end{tikzpicture}
        \centering
        \caption{Radial mesh at $0 < z < 15.$}
        \label{fig:TakedaM1C2-DDM-initial-1}
    \end{subfigure}
    }
        \resizebox{0.03\textwidth}{!}{
    \begin{subfigure}{0.4\textwidth}
        \centering
\begin{tikzpicture}
    \node at (0,0) {\includegraphics[width=4.775cm]{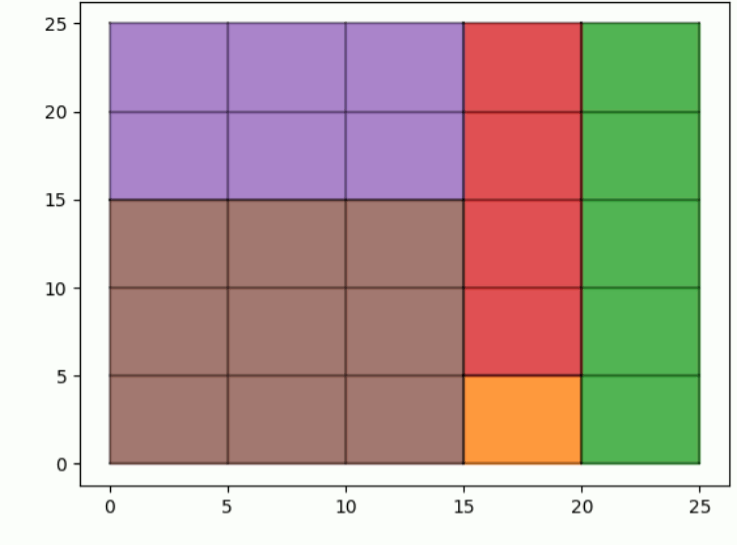}};
    \node[left] at (-2.25,0.125) {$y$};
    \node[below] at (0.2,-1.75) {$x$};
\end{tikzpicture}
        \caption{Radial mesh at $15 < z < 25.$}
        \label{fig:TakedaM1C2-DDM-initial-2}
    \end{subfigure}
    }
    \hspace{0.125mm}
        \resizebox{0.01125\textwidth}{!}{
    \begin{subfigure}{0.125\textwidth}
    \centering
    \begin{tikzpicture}[scale=0.4]
    \definecolor{couleur1}{RGB}{76, 146, 195}    
    \definecolor{couleur2}{RGB}{255, 152, 62}    
    \definecolor{couleur3}{RGB}{86, 179, 86}    
    \definecolor{couleur4}{RGB}{222, 82, 83}  
    \definecolor{couleur5}{RGB}{169, 133, 202}  
    \definecolor{couleur6}{RGB}{163, 120, 111}  

        \fill[white] (0, -1) rectangle (1, -9.5);

    \foreach \i/\couleur in {1/couleur1, 2/couleur2, 3/couleur3, 4/couleur4, 5/couleur5, 6/couleur6} {
        \fill[\couleur] (0, -\i ) rectangle (1, -\i -0.5);
        \node[right] at (1.2, -\i -0.25 ) {$\Omega_{\i}$};
    }
\end{tikzpicture}
\caption{Legend}
\end{subfigure}
}
}
    \caption{{Subdomains and initial mesh for the multi-domain approach}. 
}
    \label{fig:TakedaM1C2-DDM-initial}
\end{figure}
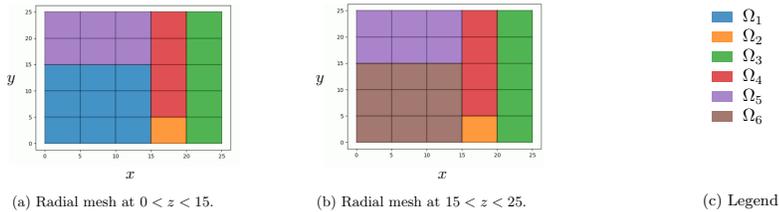

{We study again two configurations, now denoted DDM-1,DDM-2}. In the \textbf{ESTIMATE} module, the reconstruction is computed by the averaging method described in Section~\cite[Section 6.1.2]{CiDoGeMa25}. {We recall that, in the \textbf{REFINE} module, the refinement parameter $\theta_{i^*}$ is defined by subdomain: the values are given in Table~\ref{tab:theta-AMR-DDM}.}
\begin{table}[htbp]
\centering
\begin{tabular}{|c|c|c|c|}
\hline
 &  DDM-1 & DDM-2 
  \\
\hline
$\Omega_{1}$ & $0.5$ & $0.2$ 
\\
$\Omega_{2}$ & $0.5$& $0.7$  \\
$\Omega_{3}$ & $0.5$&  $0.2$  \\
$\Omega_{4}$ & $0.5$& $0.2$  \\
$\Omega_{5}$ & $0.5$&  $0.2$  \\
$\Omega_{6}$ & $0.5$&  $0.2$  \\
\hline
\end{tabular}
\caption{AMR refinement parameters {$(\theta_{i^*})_{1\leq i^*\leq 6}$, for the multi-domain approach}.}
\label{tab:theta-AMR-DDM}
\end{table}

\subsection{Numerical illustration}
\label{subsec:results}
Figure~\ref{fig:TakedaM1C2-Err_Nelem} shows the decrease of the relative error in the 
$\|\cdot\|_S$ norm in the different AMR processes. We observe that the {multi-domain approach reaches} a better accuracy with less mesh elements. Likewise, Figure~\ref{fig:TakedaM1C2-Err_Nelem} also shows the maximum of the local error indicator (\ref{LocalIndicator_mono})-(\ref{LocalIndicator_multi}) for the different AMR processes. After the AMR processes are completed, we see that there is at least a factor 6 between the {total number of mesh elements, compared to the uniform refinement}. We {also emphasize} that there is a factor 3 between the {(final) total number} of mesh elements of the DDM-2 configuration, {compared to} the MONO-1 configuration. We notice the sensitivity with respect to the refinement parameter $\theta$, as discussed in~\cite[Section 6.4]{CDM23}.

During the AMR process, Table~\ref{tab:5} shows that, for mesh elements containing a boundary facet, the relative contribution of the Robin boundary condition estimator is not dominant. We also observe that the maximum of the total estimator over the mesh elements containing a boundary facet becomes negligible compared to the maximum of the total estimator over all the mesh elements, which seems to indicate that the refinement is driven by resolving the solution accurately enough at the interface between different materials.
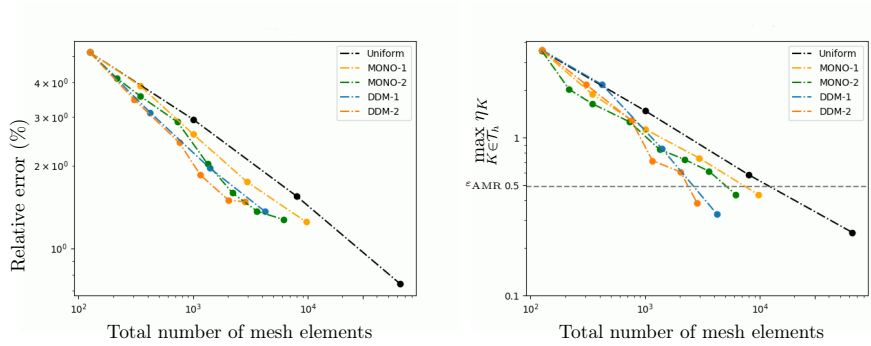
\begin{figure}[H]
\centering
\resizebox{0.8\textwidth}{!}{
\hspace{-7mm}
\resizebox{0.15\textwidth}{!}{
\begin{subfigure}[t]{0.49\linewidth}
  \raggedleft
\begin{tikzpicture}
    \node at (0,0) {\includegraphics[width=8cm]{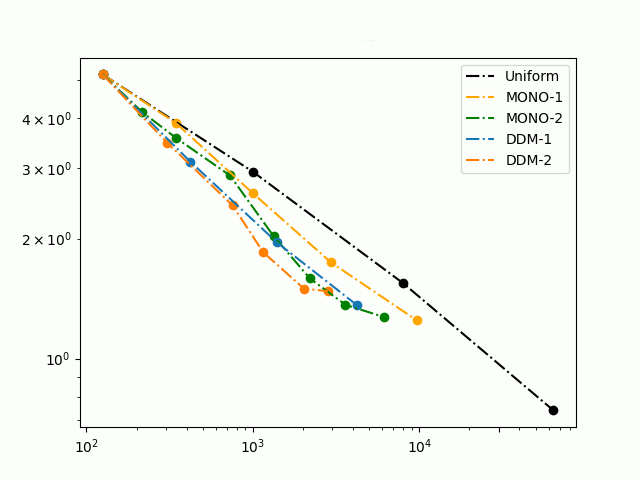}};
    \node[left, rotate=90] at (-4,1.) {Relative error (\%)};
    \node[below] at (0,-2.75) {{Total number of mesh elements}};
\end{tikzpicture}
  \label{fig:TakedaM1C2-Err_Nelem-a}
\end{subfigure}
}\hspace{5mm}
\resizebox{0.15\textwidth}{!}{
\begin{subfigure}[t]{0.49\linewidth}
  \raggedright
\begin{tikzpicture}
    \node at (0,0) {\includegraphics[width=8cm]{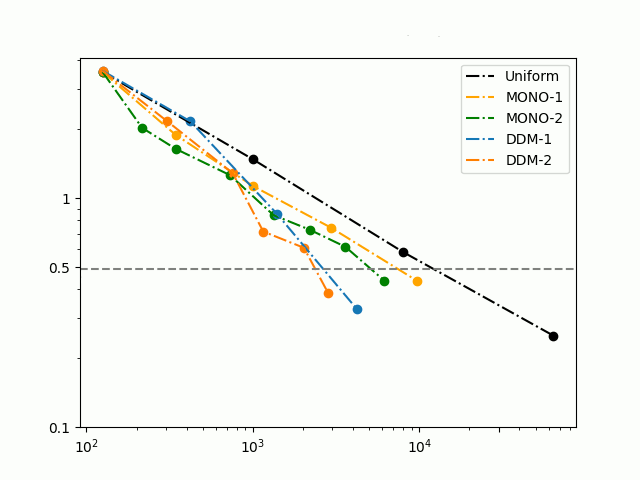}};
    \node[left, rotate=90] at (-3.7,1.4) {$\displaystyle \max_{K\in\Tcal_h} \eta_K$};
    \node[left] at (-3.3,-0.35) {\tiny $ \eps_\text{AMR}$};
    \node[below] at (0,-2.75) {{Total number of mesh elements}};
\end{tikzpicture}
  \label{fig:TakedaM1C2-Err_Nelem-b}
\end{subfigure}
}
}
\caption{Relative error in the $\|\cdot\|_S$ norm (left) and maximum of the total estimator (right) as a function of the {total number} of mesh elements.}
\label{fig:TakedaM1C2-Err_Nelem}
\end{figure}
Figures~\ref{fig:TakedaM1C2-MONO-meshfinal},~\ref{fig:TakedaM1C2-DDM-1-meshfinal} and~\ref{fig:TakedaM1C2-DDM-2-meshfinal} respectively show the final meshes of the mono-domain, DDM-1 and DDM-2 {multi-domain} configurations. We observe that refinement mostly takes place near the material interfaces. The DDM-based refinement is able to {focus on} this interface-focused refinement, which confirms its relevance when a more localized and {physics aware} refinement is required.
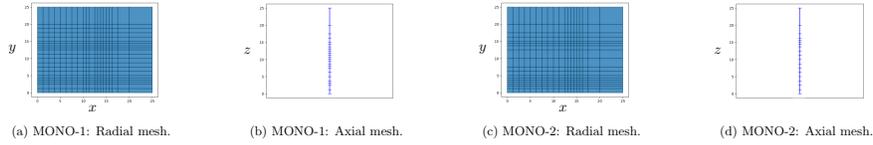
\begin{figure}[!ht]
\resizebox{\textwidth}{!}{
    \begin{subfigure}{0.4\textwidth}
        \centering
\begin{tikzpicture}
    \node at (0,0) {\includegraphics[width=4cm]{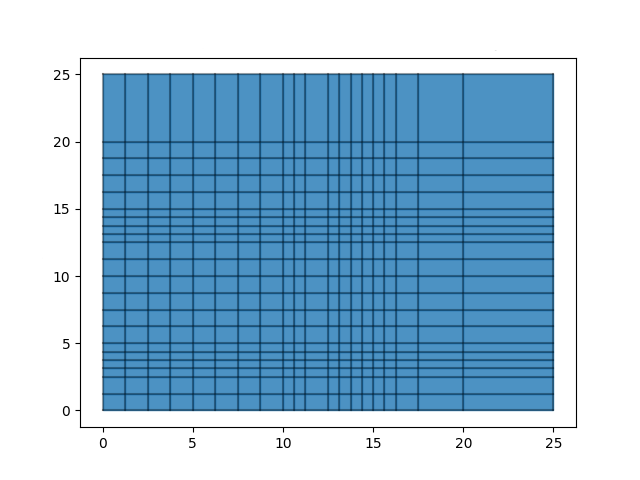}};
    \node[left] at (-1.75,0.) {$y$};
    \node[below] at (0,-1.25) {$x$};
\end{tikzpicture}
        \caption{MONO-1: Radial mesh.}
        \label{fig:TakedaM1C2-MONO-1-radial}
    \end{subfigure}
 \hspace{7mm}
    \begin{subfigure}{0.4\textwidth}
        \centering
\begin{tikzpicture}
    \node at (0,0) {\includegraphics[width=4cm]{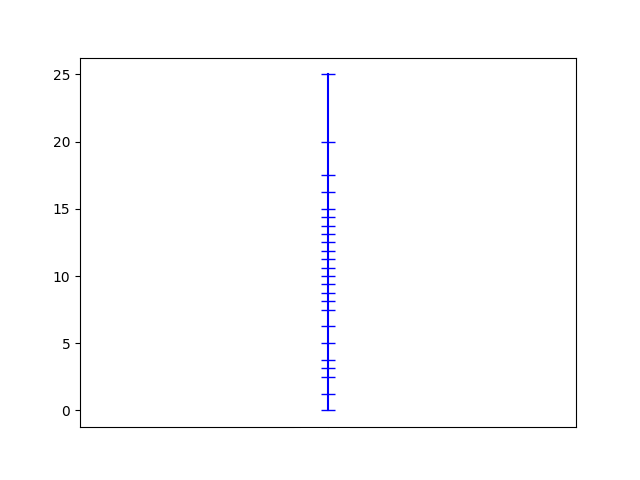}};
    \node[left] at (-1.75,0.) {$z$};
\end{tikzpicture}
        \caption{{MONO-1: Axial mesh.}}
        \label{fig:TakedaM1C2-MONO-1-axial}
    \end{subfigure}
    \hspace{7mm}
    \begin{subfigure}{0.4\textwidth}
        \centering
\begin{tikzpicture}
    \node at (0,0) {\includegraphics[width=4cm]{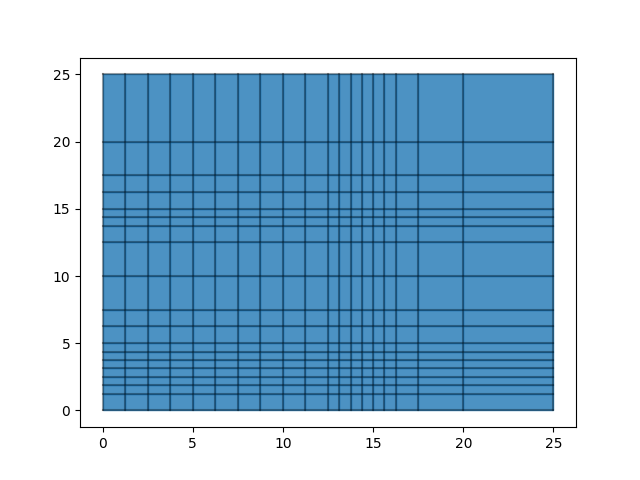}};
    \node[left] at (-1.75,0.) {$y$};
    \node[below] at (0,-1.25) {$x$};
\end{tikzpicture}
        \caption{MONO-2: Radial mesh.}
        \label{fig:TakedaM1C2-MONO-2-radial}
    \end{subfigure}
 \hspace{7mm}
    \begin{subfigure}{0.4\textwidth}
        \centering
\begin{tikzpicture}
    \node at (0,0) {\includegraphics[width=4cm]{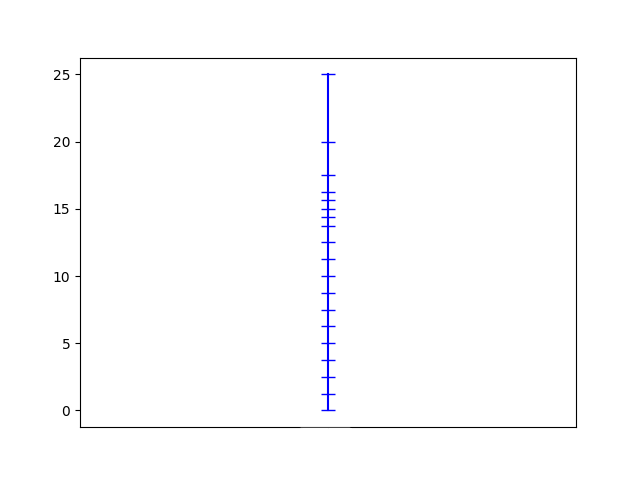}};
    \node[left] at (-1.75,0.) {$z$};
\end{tikzpicture}
        \caption{MONO-2: Axial mesh.}
        \label{fig:TakedaM1C2-MONO-2-axial}
    \end{subfigure}
    }
    \caption{{Final mesh for the mono-domain configurations.}}
    \label{fig:TakedaM1C2-MONO-meshfinal}
\end{figure}

\begin{table}[H]
  \begin{tabular}{|c|ccccc|} 
  \hline
MONO-1 & & &  & &  \\
  \hline
      { \textbf{Iteration}}                              & $|\Tcal_h|$  & $ \displaystyle \boldsymbol{\max_{K\in \Tcal_{h}} \eta_{K} }$ & $ \displaystyle \boldsymbol{\max_{F\in \mathcal{F}^e_h} \eta_{K_F} }$ & $ \displaystyle \boldsymbol{\max_{F\in \mathcal{F}^e_h} \eta_{bc,F}}$ 
      & $ \displaystyle \boldsymbol{\max_{F\in \mathcal{F}^e_h} \frac{\eta_{bc,F}}{\eta_{K_F}}}$    \\
 \hline 
0	&	125	&	3.56	&	1.28	&	0.0595	& 0.0687\\
1	&	343	&	1.89	&	0.865	&	0.0742	 & 0.153 \\
2	&	1000	&	1.13	&	0.550	&	0.0452 & 0.534	    \\
3	&	2940	&	0.744	&	0.189	&	0.0331	 & 0.592   \\
4	&	\cellcolor{blue!25}9660	&	0.437	&	0.0872	    &	0.0217	& 0.497   	\\
\hline
  \end{tabular}
  \caption{{{MONO-1 configuration: Influence of the Robin boundary condition estimator}}. 
  } 
    \label{tab:5}
\end{table}

\begin{figure}[H]
\resizebox{\textwidth}{!}{
\resizebox{0.05\textwidth}{!}{
    \begin{subfigure}{0.4\textwidth}
        \raggedleft
\begin{tikzpicture}
    \node at (0,0) {\includegraphics[width=6cm]{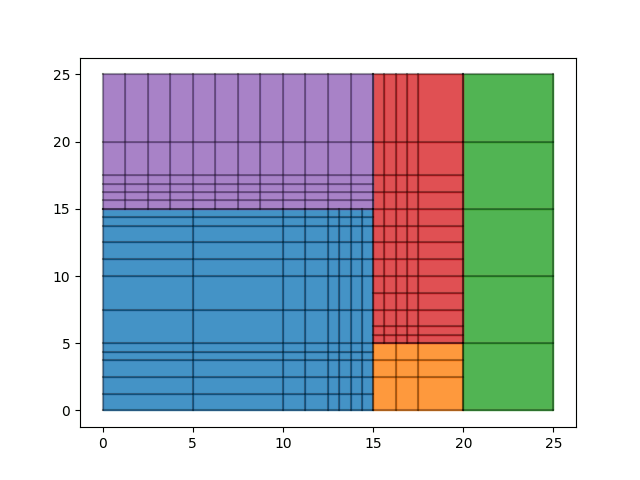}};
    \node[left] at (-2.5,0.) {$y$};
    \node[below] at (0,-2.) {$x$};
\end{tikzpicture}
        \caption{Radial mesh at  $0 < z < 15.$}
    \end{subfigure}
   }
   \resizebox{0.05\textwidth}{!}{
    \begin{subfigure}{0.4\textwidth}
        \centering
\begin{tikzpicture}
    \node at (0,0) {\includegraphics[width=6cm]{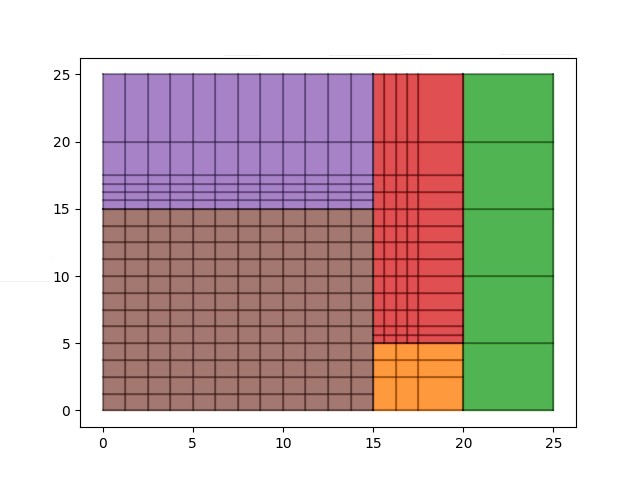}};
    \node[left] at (-2.5,0.) {$y$};
    \node[below] at (0,-2.) {$x$};
\end{tikzpicture}
        \caption{Radial mesh at $15 < z < 25.$}
    \end{subfigure}
  }
  \resizebox{0.05\textwidth}{!}{
    \begin{subfigure}{0.4\textwidth}
        \raggedright
\begin{tikzpicture}
    \node at (0,0) {\includegraphics[width=6cm]{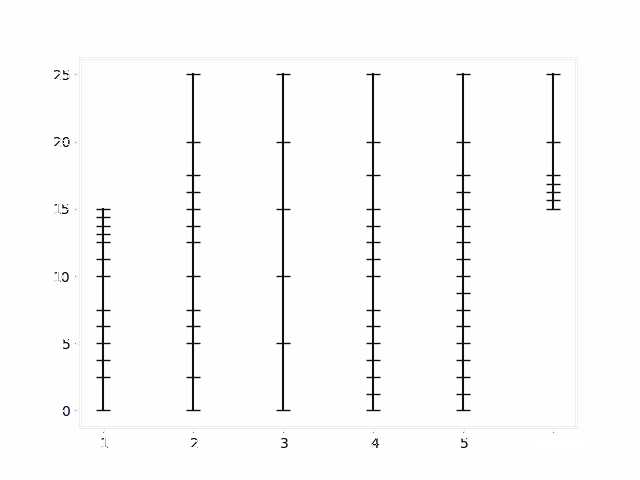}};
    \node[left] at (-2.5,0.) {$z$};
    \node[below] at (0,-2.) {Index of subdomain};
\end{tikzpicture}
        \caption{Axial mesh.}
    \end{subfigure}
    }
    }
    \caption{{Final mesh for the DDM-1 {multi-domain} configuration.}}
    \label{fig:TakedaM1C2-DDM-1-meshfinal}
\end{figure}
\begin{figure}[H]
   \resizebox{\textwidth}{!}{
   \resizebox{0.05\textwidth}{!}{
    \begin{subfigure}{0.4\textwidth}
        \raggedleft
\begin{tikzpicture}
    \node at (0,0) {\includegraphics[width=6cm]{pictures_TakedaM1C2_DDM_1_mesh-final-radial-slice0_fm.png}};
    \node[left] at (-2.5,0.) {$y$};
    \node[below] at (0,-2.) {$x$};
\end{tikzpicture}
        \caption{Radial mesh at $0 < z < 15.$}
    \end{subfigure}
}
   \resizebox{0.05\textwidth}{!}{
    \begin{subfigure}{0.4\textwidth}
        \centering
\begin{tikzpicture}
    \node at (0,0) {\includegraphics[width=6cm]{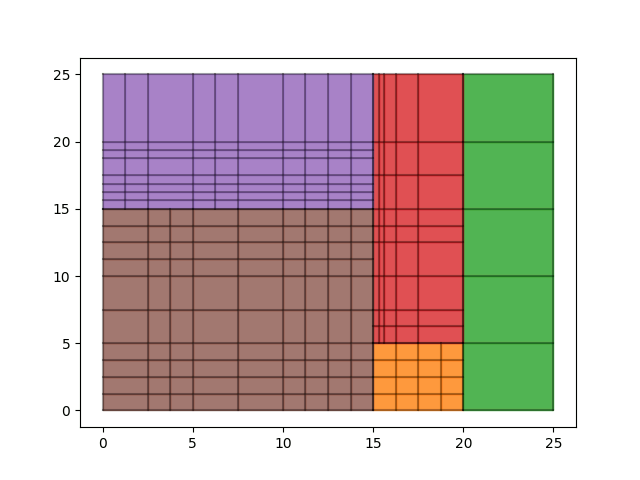}};
    \node[left] at (-2.5,0.) {$y$};
    \node[below] at (0,-2.) {$x$};
\end{tikzpicture}
        \caption{Radial mesh at $15 < z < 25.$}
    \end{subfigure}
}
   \resizebox{0.05\textwidth}{!}{
    \begin{subfigure}{0.4\textwidth}
        \raggedright
\begin{tikzpicture}
    \node at (0,0) {\includegraphics[width=6cm]{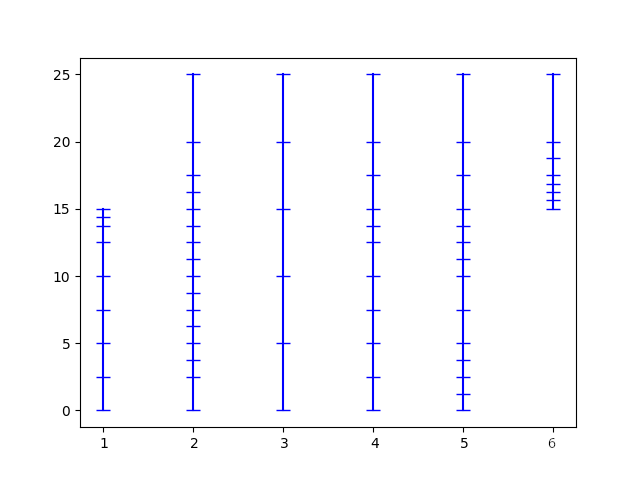}};
    \node[left] at (-2.5,0.) {$z$};
    \node[below] at (0,-2.) {Index of subdomain};
\end{tikzpicture}
        \caption{Axial mesh.}
    \end{subfigure}
    }
    }
    \caption{{Final mesh for the DDM-2 {multi-domain} configuration.}}
    \label{fig:TakedaM1C2-DDM-2-meshfinal}
\end{figure}

Finally, Table~\ref{tab:3} details the convergence of AMR {for the multi-domain configurations} on each subdomain $\Omega^*_{i^*}$, $1\leq i^*\leq 6$. We observe that the convergence of the AMR process focuses on the subdomains 1, 4, 5 and 6, which corresponds to the core and the {reflector around it}.

\begin{table}[H]
  \resizebox{\textwidth}{!}{%
  \begin{tabular}{|c|ccccccc|} 
  \hline
DDM-1 & & &  & &  & & \\
  \hline
      { \textbf{Iteration}}                              & $|\Tcal_h|$  & $ \displaystyle \boldsymbol{\max_{K\in \Tcal_{h,1}} \eta_{K} }$ & $ \displaystyle \boldsymbol{\max_{K\in \Tcal_{h,2}} \eta_{K} }$ & $ \displaystyle {\boldsymbol{\max_{K\in \Tcal_{h,3}}} \eta_{K} }$ & $ \displaystyle \boldsymbol{\max_{K\in \Tcal_{h,4}} \eta_{K} }$ & $ \displaystyle \boldsymbol{\max_{K\in \Tcal_{h,5}} \eta_{K} }$ & $ \displaystyle \boldsymbol{\max_{K\in \Tcal_{h,6}} \eta_{K} }$   \\
 \hline 
0	&	125	&	3.60	&	1.48	&	0.123	&	1.42	&	1.21	&	1.24		\\
1	&	421	&	1.44	&	0.590	&	-	    &	1.05	&	1.05	&	2.17	\\
2	&	1395	&	0.713	&	0.300	&	-	    &	0.701	&	0.806	&	0.855\\
3	&	\cellcolor{blue!25}4211	&	0.329	&	-	    &	-	    &	0.269	&	0.284	&	0.321		\\
\hline
DDM-2 & & &  & &  & & \\
  \hline
      { \textbf{Iteration}}                              & $|\Tcal_h|$  & $ \displaystyle \boldsymbol{\max_{K\in \Tcal_{h,1}} \eta_{K} }$ & $ \displaystyle \boldsymbol{\max_{K\in \Tcal_{h,2}} \eta_{K} }$ & $ \displaystyle {\boldsymbol{\max_{K\in \Tcal_{h,3}} \eta_{K} }}$ & $ \displaystyle \boldsymbol{\max_{K\in \Tcal_{h,4}} \eta_{K} }$ & $ \displaystyle \boldsymbol{\max_{K\in \Tcal_{h,5}} \eta_{K} }$ & $ \displaystyle \boldsymbol{\max_{K\in \Tcal_{h,6}} \eta_{K} }$   \\
 \hline 
0	&	125	&	3.60	&	1.48	&	0.123	&	1.42	&	1.21	&	1.24		\\
1	&	305	&	1.63	&	0.591	&	-	    &	1.19	&	1.78	&	2.17	\\
2	&	756	&	0.907	&	0.208	&	-	    &	1.29	&	0.982	&	1.01\\
3	&	1155	&	0.715	&	-	    &	-	    &	0.634	&	0.552	&	0.530		\\
4	&	2027	&	0.376	&	-	    &	-	    &	0.357	&	0.607	&	0.572		\\
5	&	\cellcolor{blue!25}2833	&	 -	    &	-	    &	-	    &	-	    &	0.386	&	0.235		\\
\hline
  \end{tabular}
  } 
  \caption{{AMR convergence for the multi-domain approach}. 
  } 
    \label{tab:3}
\end{table}

\begin{remark} On another perspective, one of the well-known advantages of the multi-domain approach is to allow for parallelization (not implemented here). \end{remark}

\section{Conclusion}
In this manuscript, we derive {\em a posteriori} estimates associated to an appropriate norm for the numerical solution of the multigroup neutron simplified transport equation {in mixed form} with vacuum boundary conditions imposed on (part of) the boundary. 
 We propose {\em a posteriori} estimators that are both reliable and locally efficient, {which requires the design of a specific component of the estimator to handle the vacuum boundary condition}.\\
We explicitly state the {\em a posteriori} estimates in the specific case of the multigroup neutron diffusion equation. We extend {\em a posteriori} estimates associated to different norms for the DD+$L^2$ jumps method, {a multi-domain approach}. \\

\bibliographystyle{plain}
\bibliography{Bibliography.bib}

\appendix
\section{{A model with} mixed boundary conditions}
\label{sec:appendix_mixed}
In this section, we describe how a posteriori estimation theory can be extended to the case where mixed boundary conditions are imposed on the boundary. For the sake of readability, we keep the same notations as in the manuscript. We split the boundary into three disjoint, open parts such that $\pa \Omega = \overline{\Gamma_D}\cup\overline{\Gamma_V}\cup\overline{\Gamma_N}$, where {$\Gamma_D$, $\Gamma_V$, $\Gamma_N$} are (possibly non-empty) Lipschitz submanifolds of $\partial\Omega$. In mixed form, the neutron multigroup SP$_N$ problem writes:
\begin{equation*}
\left\{\begin{array}{l}
\mbox{Find $(\pvec,\phi)\in{\Qududvec{}_m(\Omega)\times {\Vudud}{}_m}$ such that}\cr
\T_o\,\pvec\,+\,\H\grad\phi=0\mbox{ in }\Omega,\cr
\H^T\dive\pvec\,+\,\T_e\phi=S_{f}\mbox{ in }\Omega,\cr
{-\H^T\pvec\cdot \nvec + \Gamma_e\phi=0\mbox{ on }\Gamma_V,} \cr
{\phi=0\mbox{ on }\Gamma_D,}\cr
{\pvec\cdot \nvec =0\mbox{ on }\Gamma_N,}
\end{array}\right.
\end{equation*}
where 
\[\begin{array}{rcl}
\Qvec_m(\Omega) &=& \left\{\, \qvec\in\Hvec(\dive,\Omega) \, | (\qvec\cdot\nvec)_{|_{\Gamma_V}}\in L^2(\Gamma_V), \, (\qvec\cdot\nvec)_{|_{\Gamma_N}} = 0 \right\},\cr
& & \|\qvec\|_{\Qvec_m(\Omega)}=\left(\|\qvec\|_{\Hvec(\dive,\Omega)}^2\,+\,\|\qvec\cdot\nvec\|_{0,\Gamma_V}^2\right)^{1/2};\cr
V_m &=& \{\psi\in H^1(\Omega) \, | \psi_{|_{\Gamma_D}} = 0 \}. 
\end{array}\]
We also introduce
\[ \Xcal_m = \left\{\,(\qvec,\psi)\in\Qvec_m(\Omega)\times L^2(\Omega)\right\}\,,\ \|(\qvec,\psi)\|_{\Xcal_m}=\left(\|\qvec\|_{\Qvec_m(\Omega)}^2\,+\,\|\psi\|_{0,\Omega}^2\right)^{1/2}\,. \]
Following~\cite[Section 4.1]{JCi13} or the Appendix of \cite{CiJK17} for a justification of the integration by parts formula in the case of mixed boundary conditions, one can check that the corresponding bilinear form is defined for all $(\pvec,\phi),(\qvec,\psi) \in \Xcalmg{}_m$ by
\begin{align*}
((\pvec,\phi),(\qvec,\psi)) &\mapsto 
-(\T_o\,\pvec,\qvec)_{0,\Omega}
+(\phi,\H^T\dive\qvec)_{0,\Omega}
+(\psi,\H^T\dive\pvec)_{0,\Omega}\\
&
\qquad +(\T_e\,\phi,\psi)_{0,\Omega}
{- (\tilde{\Gamma}_e(\pvec\cdot\nvec), (\qvec\cdot\nvec))_{0,\Gamma_V}},
\end{align*}
and the variational formulation is similar to (\ref{eq:VF-3-PC}). The discrete, finite-dimensional, conforming spaces $\Xcal_{m,h} = \Qvec_{m,h}\times L_h$ are simply built with $\Qvec_{m,h} = \Qvec_h \cap \Qvec_m(\Omega)$, where $\Qvec_h$ and $L_h$ are introduced in Section~\ref{ss-sec-FEM-PC}. The conforming discretization of the variational formulation is classical (and omitted here). The discrete solution is denoted $\zeta_h=(\pvec_h,\phi_h)$. Then, one can prove that Theorems~\ref{th:VF-1-PC} and~\ref{th:disc+udisc-PC} are also valid in the case of mixed boundary conditions, using the same maps as defined in their respective proof.

Next, let $\tilde{\zeta}_h= (\pvec_h,\tilde{\phi}_h) \in \Qududvec{}_{m,h}\times \Vudud_m$ be a reconstruction of $\zeta_h$. Due to the mixed boundary conditions, the definition of the interpolation is slightly modified. We detail the case of the averaging operator of the neutron flux where 
$\mathcal{I}_{av} : \Pudud_k(\mathcal{T}_h) \to \Pudud_{k+1}(\mathcal{T}_h) \cap V_m$  is such that $\forall \phi_h\in\Pudud_k(\mathcal{T}_h),$
\begin{equation*}
  \forall a\in\mathcal{V}_h^{k+1},\quad \mathcal{I}_{av}(\phi_h)(a)=
 \left\{\begin{aligned}
 &{\frac{1}{|\mathcal{T}_a|} \displaystyle \sum_{K\in \mathcal{T}_a} ({\Gamma_e^{-1}\H^T(\pvec_h\cdot\nvec)}){}_{|K}(a) \quad \text{ if }a\in\overline{\Gamma_V}\setminus\overline{\Gamma_D},}\\ 
 &\frac{1}{|\mathcal{T}_a|} \displaystyle \sum_{K\in \mathcal{T}_a} \phi_h{}_{|K}(a)\quad \text{ otherwise.}
 \end{aligned}\right.
\end{equation*}
We finally define the strenghtened norm as in (\ref{strenghtened-norm}), the only difference being that the sum over faces is now taken for $F \in \mathcal{F}^e_h\cap\overline{\Gamma_V}$, while the local $|\cdot|_{+,K}$-norm remains defined as in (\ref{local-norm}).
\begin{theorem}[reliability]\label{theorem:Mixed_norm+K}
{Let $\zeta$ be the solution to~\eqref{eq:VF-3-PC}. With the same notation as in definition~\ref{definition:estimators}}, one has the estimate
\begin{align*}
&|\zeta-\tilde{\zeta}_h|_{+,K}\leq \left({\eta}^2_{r,K} +\sum_{K'\in N(K)}\eta^2_{f,K'}{+\sum_{F\in {\mathcal{F}_h^e}\cap \overline{\Gamma_V}\cap \pa K}\eta_{bc,F}^2}\right)^{1/2}.
\end{align*}
\end{theorem}
\begin{proof}
The proof follows the same pattern as the proof of Theorem~\ref{theorem:norm+K}.
\end{proof}
\begin{theorem}
[efficiency]
\label{Mixed_thm_apost_error_est}
Let Assumption~\ref{assumption:locality_polynomial} be fulfilled. 
For $K \in \Tcal_h$, let ${\eta}_{r,K}$ and ${\eta}_{f,K}$ be the residual and flux estimators respectively given by~\eqref{eq:def_residual_estimator}, and~\eqref{eq:def_flux_estimator}. The following estimates hold true
\begin{align*}
{\eta}_{r,K}&\leq \mathtt{c}\,{ \left(\frac{\delta^{max}_{e,K}}{\delta^{min}_{e,K}}\right)^{1/2}}\,|\zeta-\tilde{\zeta}_h|_{+,K}
,\\ 
{{\eta}_{f,K}}&{\leq { \mathtt{C} \left(\frac{\delta^{max}_{o,K}}{\delta^{min}_{o,K}}\right)^{1/2}}\,|\zeta-\tilde{\zeta}_h|_{+,K}}
,
\end{align*}
where {$\mathtt{c}$ and $\mathtt{C}$ are constants which depend} only on the polynomial degree of $S_f$, $\T_o$, $\T_e$  {and $\tilde{\phi}_h$}, $d$, and the shape-regularity parameter $\kappa_K$. \\
{For $F \in \mathcal{F}^e_h\cap \overline{\Gamma_V}$, let ${\eta}_{bc,F}$ be the Robin boundary condition estimator given by~\eqref{eq:def_bc_estimator}. 
The following estimate holds true
\begin{align*}
{{\eta}_{bc,F}}&{\leq  \mathsf{C} \,|\zeta-\tilde{\zeta}_h|_{+,K_F}}
,
\end{align*}
where $h_{\perp F}$ is the size of the $F$-transverse part of the {mesh element $K_F$} containing $F$ in its facets, $\mathsf{C}$ is a constant which depends only on the polynomial degree of $S_f$,  $\T_e$  {and $\tilde{\phi}_h$}, $d$, and the shape-regularity parameter $\kappa_{K_F}$.
}
\end{theorem}
\begin{proof}
The proof is identical to the proof of Theorem~\ref{thm_apost_error_est}.
\end{proof}

\end{document}